\documentclass[10pt, reqno]{amsart}
\usepackage{amsfonts}
\usepackage{amsmath}
\usepackage{pifont}
\usepackage{amssymb}
\usepackage{amsthm}
\usepackage{hyperref}

\newtheorem{thm}{Theorem}
\newtheorem{prop}[thm]{Proposition}
\newtheorem{lemma}[thm]{Lemma}
\newtheorem{cor}[thm]{Corollary}
\numberwithin{thm}{section}

\newtheorem*{thm_one}{Theorem 1}
\newtheorem*{thm_two}{Theorem 2}
\newtheorem*{thm_three}{Theorem 3}
\newtheorem*{main_thm_part_one}{Main Theorem (Part I)}
\newtheorem*{thm_five}{Theorem 5}
\newtheorem*{thm_six}{Theorem 6}
\newtheorem*{conj}{Conjecture}
\newtheorem*{thm_seven}{Theorem 7}
\newtheorem*{thm_eight}{Theorem 8}
\newtheorem*{cor_nine}{Corollary 9}
\newtheorem*{main_thm_part_two}{Main Theorem (Part II)}
\newtheorem*{rem}{Helton-Vinnikov Theorem}

\newcommand{\shrink}[1]{{\scriptstyle {\textstyle {#1}}}}
\newcommand{\smfrac}[2]{\shrink{\frac{#1}{#2}}}

\newcommand{\reals }{\mathbb{R} }

\newcommand{\ppi}[1]{p^{({#1})}}
\newcommand{\ppie}[2]{p^{({#1})}_{{#2}}}
\newcommand{\qie}[2]{q^{({#1})}_{{#2}}}

\newcommand{\plus}{{\scriptscriptstyle +}}

\newcommand{\Lambdap}{\Lambda_{\plus}}
\newcommand{\Lambdapp}{\Lambda_{\plus \plus}} 
\newcommand{\Lambdape}[1]{\Lambda_{\plus,{#1}}}
\newcommand{\Lambdappe}[1]{\Lambda_{\plus \plus,{#1}}}

\newcommand{\Lambdaiep}[2]{\Lambda^{({#1})}_{\plus,{#2}}}
\newcommand{\Lambdaiepp}[2]{\Lambda^{({#1})}_{\plus \plus,{#2}}}

\newcommand{\hp}{\mathrm{HP}} 
\newcommand{\hpie}[2]{\hp_{{#2}}^{({#1})}} 
\newcommand{\opt}{\mathrm{Opt} }
\newcommand{\optie}[2]{\mathrm{Opt}_{{#2}}^{({#1})}}
\newcommand{\feas}{\mathrm{Feas} } 
\newcommand{\feasie}[2]{\mathrm{Feas}^{({#1})}_{{#2}} } 
\newcommand{\val}{\mathrm{val} } 
 
\newcommand{\xie}[2]{x_{{#2}}^{({#1})} } 
\newcommand{\sie}[2]{s_{{#2}}^{({#1})} } 
\newcommand{\yie}[2]{y_{{#2}}^{({#1})} } 
\newcommand{\swath}{\mathrm{Swath}} 
\newcommand{\core}{\mathrm{Core}} 
\renewcommand{\int}{\mathrm{int}}
\newcommand{\relint}{\mathrm{relint}} 
 
\newcommand{\mult}{\mathrm{mult} }

\newcommand{\mie}[2]{\mathrm{M}_{{#2}}^{({#1})}} 
\newcommand{\nie}[2]{\mathrm{N}_{{#2}}^{({#1})}}
 
\newcommand{\ray}{\mathrm{ray}} 

\newcommand{\sep}{ \vspace{1mm} \begin{center} \ding{70} \ding{70} \ding{70} \end{center} \vspace{1mm}  }

\newcommand{\cpath}{\mathrm{Central \, Path}} 
\newcommand{\ecpath}{\mathrm{ECentral \,  Path}}

\begin{document}

\title[Central Swaths]{\large{Central Swaths} \\ {\small (A Generalization of the Central Path)}}
\author[J. Renegar]{James Renegar}
\address{School of Operations Research and Information Engineering,
 Cornell University, Ithaca, NY, U.S.}

 \thanks{Research supported in part by NSF Grant \#CCF-0430672.}
\thanks{Thanks to Chek Beng Chua and Yuriy Zinchenko for many helpful conversations.  Deep gratitude also goes to the referees, whose extensive comments led to the paper being significantly restructured, and led to the (motivational) exposition being considerably expanded especially in the proofs.}

 \begin{abstract} 
We develop a natural generalization to the notion of the central path -- a notion that lies at the heart of interior-point methods for convex optimization.  The generalization is accomplished via the ``derivative cones'' of a ``hyperbolicity cone,'' the derivatives being direct and mathematically-appealing relaxations of the underlying (hyperbolic) conic constraint, be it the non-negative orthant, the cone of positive semidefinite matrices, or other.  

We prove that a dynamics inherent to the derivative cones generates paths always leading to optimality, the central path arising from a special case in which the derivative cones are quadratic.  Derivative cones of higher degree better fit the underlying conic constraint, raising the prospect that the paths they generate lead to optimality quicker than the central path. 
\end{abstract}

\maketitle  

\numberwithin{equation}{section}

\section{{\bf  Introduction}} \label{s.a}

Let $ {\mathcal E} $ denote a finite-dimensional Euclidean space and let  $ p: {\mathcal E}  \rightarrow \reals $ be a hyperbolic polynomial, that is, a homogeneous polynomial for which there is a designated direction vector $ e $ satisfying $ p(e) > 0 $ and having the property that for all $ x \in {\mathcal E}   $, the univariate polynomial $ t \mapsto p(x+te) $ has only real roots.  Thus, $ p $ is ``hyperbolic in direction $ e $.''

Let $ \Lambdapp $ denote the hyperbolicity cone  -- the connected component of $ \{ x: p(x) > 0 \} $ containing $ e $.  Let $ \Lambdap $ be the closure.

A simple example is $ p(x) = \prod_j x_j $ and $ e = (1,\ldots,1) $, in which case $ \Lambdapp $ is the strictly positive orthant and $ \Lambdap $ is the non-negative orthant.  Perhaps the most fundamental example, however, is $ p(X) = \det(X) $, where $ X $ ranges over $ n \times n $  symmetric matrices and $ e = I $, the identity matrix.  Here, $ \Lambdapp $ is the cone of (strictly-)positive definite (pd) matrices and $ \Lambdap $ is the positive semidefinite (psd) cone. 

G\r{a}rding \cite{garding}  showed for each hyperbolic polynomial $ p $ that every $ \hat{e} \in \Lambdapp $ is a hyperbolicity direction, i.e., for every $ x $, all of the roots of $ t \mapsto p(x+t \hat{e}  ) $ are real.      One of several remarkable corollaries G\r{a}rding established is that $ \Lambdapp $ is convex. (Of course $ \Lambdap $ is thus convex, too.) (See \S2 of \cite{renegar}   for simplified proofs.) 

The combination of convexity and rich algebraic structure make hyperbolicity cones promising objects for study in the context of optimization, as was first made evident by G\"{u}ler \cite{guler}, who developed a rich theory of interior-point methods for hyperbolic programs, that is, for problems of the form
\[ \begin{array}{rl}
\min & c^*x \\
\textrm{s.t.} & Ax = b \\
& x \in \Lambdap  \end{array} \]
-- linear programming, second-order programming and semidefinite programming being particular cases.  Key to G\"{u}ler's development is that the function $ x~\mapsto~-~\ln~p(x) $ is a self-concordant barrier for $ \Lambdap $; thus the general theory of Nesterov and Nemirovski \cite{nn} applies. 

A primary purpose of the present paper is to use the viewpoints provided by hyperbolic programming to develop a natural generalization to the notion of the central path\footnote{Central Path :=  $ \{ x( \eta ) : \eta > 0 \} $ where $ x(\eta ) $ solves $ \min_x \eta c^* x - \ln p(x) $, s.t. $ Ax = b $, $ x \in \Lambdapp $.} (a notion that  lies at the heart of interior-point method theory).  This is accomplished via derivative cones, which are direct relaxations of the underlying convex conic constraint (be it the non-negative orthant, the cone of positive semidefinite matrices, \ldots).

However, perhaps more important than the ``natural generalization to the notion of the central path'' is our ``use (of) the viewpoints provided by hyperbolic programming'' in developing the generalization.  Indeed, it is our conviction that even if results about hyperbolic programming never find application more general than linear programming and semidefinite programming, the setting of hyperbolic programming is favorable for engendering intriguing algorithmic ideas that otherwise would have been unrealized (or at least considerably delayed).  

Familiarity with the central path is not required to readily understand our results.  (The central path simply provides an initial anchor with which many readers {\em are} familiar.)

The literature focusing on hyperbolic polynomials is relatively small but its growth is accelerating and its quality in general is distinctly impressive.  Although the nature of our results is such that during the development we have occasion to cite only a few works, we take the opportunity before beginning the development to draw the reader's attention to the bibliography, which includes a variety of notable papers appearing in recent years.  In particular, an appreciation of the breadth and quality of research ideas surrounding hyperbolic polynomials can be fostered by browsing \cite{bgls}, \cite{bb}, \cite{gurvits}, \cite{hl}, and \cite{hv}.

\section{{\bf  Overview of Results}} \label{s.b}

Let $ \phi $ be a univariate polynomial all of whose coefficients are real. Between any two real roots of $ \phi $ there lies, of course, a root of $ \phi' $.  Consequently, because $ \phi' $ is of degree one less than the degree of $ \phi $, a simple counting argument shows that if all of the roots of $ \phi $ are real, then so are all of the roots of $ \phi' $.

In particular, if $ \phi(t) := p(x+te) $ where $ p $ is a polynomial hyperbolic in direction $ e $ (and where $ x $ is an arbitrary point), then all the roots of $ t \mapsto \phi'(t) = \smfrac{d}{dt} p(x+te) = Dp(x+te)[e] $ are real, where $ Dp(x+te) $ denotes the differential of $ p $ at $ x + te $.  Hence, the polynomial $ p'_e(x) := Dp(x)[e]  $ is, like $ p $, hyperbolic in direction $ e $.  

For example, if $ p(x) = \prod_j x_j $ and all coordinates of $ e $ are nonzero, then $ p_e'(x) = \sum_i e_i \prod_{j \neq i} x_j $ is hyperbolic in direction $ e $.

We refer to $ p'_e $ as the ``derivative polynomial (in direction $ e $),'' and denote its hyperbolicity cone by $ \Lambdappe{e}' $.  The fact that for every $ x $, the roots of $ t \mapsto  p_e'(x+te)  $ lie between the roots of $ t \mapsto p(x+te) $ is readily seen to imply  $ \Lambdapp \subseteq \Lambdappe{e}' $ -- in words, $ \Lambdappe{e}' $ is a {\em relaxation} of $ \Lambdapp $ (see \S4 of \cite{renegar} for a full discussion).  

Of course one can in turn take the derivative in direction $ e $ of the hyperbolic polynomial $ p'_e $, thereby obtaining yet another polynomial -- $ (p'_e)'_e(x) = D^2p(x)[e,e] $ -- hyperbolic in direction $ e $. Letting $ n $ denote the degree of $ p $, repeated differentiation in direction $ e $ results in a sequence of hyperbolic polynomials 
\[  \ppie{1}{e} = p_e', \, \ppie{2}{e}, \ldots, \ppie{n-1}{e} \; , \]
  where $ \deg( \ppie{i}{e}) = n-i $.  (For convenience, let $ \ppie{0}{e} := p $.) The associated hyperbolicity cones $ \Lambdaiepp{i}{e} $ -- and their closures $ \Lambdaiep{i}{e} $ -- form a nested sequence of relaxations of the original cone:
\[  \Lambdap = \Lambdaiep{0}{e}  \subseteq \Lambdaiep{1}{e} \subseteq \Lambdaiep{2}{e} \subseteq \cdots \subseteq \Lambdaiep{n-1}{e} \; . \]
The final relaxation, $ \Lambdaiep{n-1}{e} $, is a halfspace, because $ \ppie{n-1}{e} $ is linear.

The cones become tamer as additional derivatives are taken.  The halfspace $ \Lambdaiep{n-1}{e} $ is as tame as a cone can be, but extremely tame also is the second-order cone $ \Lambdaiep{n-2}{e} $ -- no cone with curvature could be nicer.  As one moves along the nesting towards the original cone $ \Lambdap $, the boundaries gain more and more corners.  For example, when $ p(x) = \prod_j x_j $ and all coordinates of $ e $ are positive, the boundary $ \partial \Lambdaiep{i}{e} $ contains all of the non-negative orthant's faces of dimension less than $ n-i $ (hence {\em lots} of corners when $ i $ is small and $ n $ large).  On the other hand, everywhere else, $ \partial \Lambdaiep{i}{e} $ has nice curvature properties (no corners), as is reflected in the following motivational theorem pertaining to every hyperbolicity cone whose closure is regular (i.e., has nonempty interior and contains no subspace other than the origin). 

\hypertarget{targ_thm_one}{}
\begin{thm_one}  
Assume $ \Lambdap $ is regular and $ 0 \leq i \leq n-2 $.
\begin{enumerate}

\item  The intersection $ \Lambdap \cap \partial \Lambdaiep{i}{e} $ is independent of $ e \in \Lambdapp $ \\
(thus, a face of $ \Lambdap $ which is a boundary face of $ \Lambdaiep{i}{e} $ for some $ e \in \Lambdapp $ is a boundary face for all $ e \in \Lambdapp $).  

\item If $ e \in \Lambdapp $ then any boundary face of $ \Lambdaiep{i}{e} $ either is a face of $ \Lambdap $ \\ or is a single ray  contained in $ \Lambdaiepp{i+1}{e} \setminus \Lambdaiep{i-1}{e} $. 
\end{enumerate} 
\end{thm_one}

We show in \S\ref{s.d}  that the theorem is a consequence of results from \cite{renegar}.  (In order to make the present section inviting to a broad audience, nearly all proofs are delayed.)

Before moving to discussion of hyperbolic programs, we record a characterization of the derivative cones that is useful both conceptually and in proofs: 
\begin{equation}  \label{e.a.a}
 \Lambdaiep{i}{e} = \{ x:\ppie{j}{e}(x) \geq  0 \textrm{ for all $j = i, \ldots, n-1 $} \} \; . 
\end{equation}
(This is immediate from Proposition 18 and Theorem 20 of \cite{renegar}.)

\sep

Consider a hyperbolic program
\[  
 \left. \begin{array}{rl}
\min & c^*x \\
\mathrm{s.t.} & Ax = b \\
& x \in \Lambdap \end{array} \quad \right\} \, \hp
\]
and its derivative relaxations in direction $ e $,
\[  
 \left. \begin{array}{rl}
\min & c^*x \\
\mathrm{s.t.} & Ax = b \\
& x \in \Lambdaiep{i}{e} \end{array} \quad \right\} \, \hpie{i}{e} \quad \textrm{($ i=1, \ldots, n-1 $)} \; .  
\] 
(Strictly speaking, ``$\min$'' should be replaced with ``$\inf$,'' but we focus on instances where a minimizer exists.)  The optimal values for the derivative relaxations $ \hpie{i}{e} $ form a decreasing sequence in $ i $, due to the nesting of the derivative cones.

Let $ \feas $ (resp., $ \feasie{i}{e} $) denote the feasible region of $ \hp $ (resp., $ \hpie{i}{e} $) -- the set of points satisfying the constraints.  Let $ \opt $ ($ \optie{i}{e} $) denote the set of optimal points -- a.k.a. optimal solutions -- and let $ \val $ denote the optimal value of $ \hp $.

\underline{We assume $ b \neq 0 $} (thus, the origin is infeasible, and the feasible sets are not cones), \underline{$ A $ is surjective} (i.e., onto),  \underline{and $ c^* $ is not in the image of $ A^* $} (otherwise every feasible point would be optimal).

\underline{We assume $ \Lambdap $ is a regular cone}.  Then, for $ 1 \leq i \leq n-2 $, $ \Lambdaiep{i}{e} $ also is regular (\cite{renegar}, Proposition 13).

From these assumptions and Theorem \hyperlink{targ_thm_one}{1}(B) immediately follows a fact that will play a critical role:
\begin{equation}  \label{e.a.b}
\left. \begin{array}{l}
 \textrm{If $ 1 \leq i \leq n-2 $, then either} \\
$ \textrm{ ~} $ \quad   \textrm{$ \optie{i}{e} = \emptyset  $ ,} \\
$ \textrm{ ~} $ \quad \textrm{$ \optie{i}{e} = \opt $ , or} \\ $ \textrm{ ~} $ \quad \textrm{$ \optie{i}{e} $ consists of a single point} \\
$ \textrm{ ~} $ \qquad \quad \textrm{and the point is contained in $  \relint(\feasie{i+1}{e}) \setminus  \feasie{i-1}{e} $ ,} \end{array} \quad \right\} \end{equation}
where ``$ \relint $'' denotes relative interior\footnote{The ``relative interior'' of a convex set $ S \subseteq {\mathcal E} $ is the interior of $ S  $ when considered as a subset in the smallest affine space containing $ S $ (where the affine space inherits the topology of the Euclidean space $ {\mathcal E}$).}.

Thus, for each $ 1 \leq i \leq n-2 $, the cone $ \Lambdapp $ is naturally partitioned into three regions, one consisting of derivative directions $ e $ for which $ \optie{i}{e} = \opt $, a second consisting of directions  for which $ \optie{i}{e} = \emptyset $, and the third consisting of directions for which $ \optie{i}{e} $ consists of a single point lying outside the feasible region for the original optimization problem $ \hp $.  We associate names with this partitioning of $ \Lambdapp $, but before doing so, we introduce a restriction.

We shall only be concerned with derivative directions $ e $ satisfying $ Ae = b $ (indeed, key arguments rely heavily on $ A(e-x) = 0 $ for $ x \in \feasie{i}{e} $).  Thus, the derivative directions we consider satisfy $ e \in \feasie{i}{e} $ for all $ i $ -- in particular, $ \hpie{i}{e} $ is ``strictly'' feasible, as $ e \in \Lambdapp \subseteq \Lambdaiepp{i}{e} $.  

We distinguish two sets of derivative directions for $ 0 \leq i \leq n-1 $:
\begin{quote}
The {\em $ i^{\mathrm{th}} $ central swath} is the set
\[     \swath(i) :=  \{ e \in \Lambdapp:  Ae = b \textrm{ and }  \optie{i}{e} \neq \emptyset \} \; , \]   
and the set of {\em core} derivative directions is defined by 
\[  \core(i) := \{ e \in \swath(i): \optie{i}{e} = \opt \} \; . \]
When $ 1 \leq i \leq n-2 $ and $ e \in \swath(i) \setminus \core(i) $, we use $ \xie{i}{e} $ to denote the unique point in $ \optie{i}{e} $ (unique by (\ref{e.a.b})).
\end{quote}
For reference, we note that from (\ref{e.a.b}),
\begin{equation}  \label{e.a.c}
 \xie{i}{e} \in  \Lambdaiepp{i+1}{e} \cap \partial \Lambdaiep{i}{e} \; . 
\end{equation}

Trivially, if $ \opt \neq \emptyset $ (resp., $ = \emptyset $), then $ \swath(0) = \core(0) = \relint(\feas) $ (resp., $ = \emptyset $), that is, the zeroth swath coincides precisely with the relative interior of $ \hp $'s feasible region (resp., is the empty set).  More interestingly, the swaths and cores are nested:
\begin{gather*} 
    \swath(0) \supseteq \swath(1) \supseteq \cdots \supseteq \swath(n-1) \\ \label{}
   \core(0) \supseteq \core(1) \supseteq \cdots \supseteq \core(n-1)
\end{gather*}
For the cores, the nesting is an easy consequence of the reverse nesting 
\[  \feasie{0}{e} \subseteq \feasie{1}{e} \subseteq \cdots \subseteq \feasie{n-1}{e} \; . \]
These reverse nesting also provide the crux in proving the nesting of the swaths, a proof we defer to \S\ref{s.d}. 

A consequence of the nesting of swaths is that if any swath is nonempty, then $ \swath(0) \neq  \emptyset  $ -- equivalently, $ \opt \neq \emptyset $.

Whereas a path is narrow (one-dimensional), swaths can be broad, just as $ \swath(i) $ typically fills much of the feasible region for $ \hp $ when $ i $ is small.  But why do we use the terminology ``{\em central} swaths'' rather than simply ``swaths''?  The following elementary theorem (proven in \S\ref{s.d}) gives our first reason.

\hypertarget{targ_thm_two}{}
\begin{thm_two}   \quad  $  \swath(n-1) = \mathrm{Central \, Path} \; .  $ 
\end{thm_two}

The central path is fundamental in the literature on interior-point methods.  The path leads to optimality.  Most of the algorithms follow the path, either explicitly or implicitly.   A foremost goal of the present paper is to show that not only does the central path lead to optimality, but {\em all} central swaths lead, in a natural manner, to optimality.  We show, in particular, that through each point $ e \in \swath(i) $ (for $ 1 \leq i \leq n-2 $), there is naturally defined a trajectory which leads from $ e $ to optimality; moreover, the trajectory remains within $ \swath(i) $ until optimality is reached. An intriguing possibility is that 
for small values of $ i $, the trajectory might lead to optimality ``more quickly'' than the central path.  (Motivation for this possibility will become clearer as the reader proceeds.) 

For $ 1 \leq i \leq n-2 $, consider the idealized setting in which for derivative directions $ e \in \swath(i) $, an exact optimal solution for $ \hpie{i}{e} $ can be computed.  If the optimal solution lies in $ \Lambdap $, then clearly it lies in $ \opt $, the set of optimal solutions for the original optimization problem $ \hp $.  In this case our goal of solving $ \hp $ has been accomplished.  On the other hand, if the optimal solution does not lie in $ \Lambdap $, then $ e \in \swath(i) \setminus \core(i) $, and the optimal solution is the unique point $ \xie{i}{e} $ in $ \optie{i}{e} $.  In this case how can we move towards solving $ \hp $?  How can we construct a trajectory $ t \mapsto e(t) $ for which $ e(0) = e $ and such that either the trajectory converges to $ \opt $ or the path $ t \mapsto \xie{i}{e(t)} $ converges to $ \opt $ (or both)? 

An apparently easier task would be to create a trajectory $ t \mapsto e(t) $ for which $ t \mapsto c^* e(t) $ is monotonically decreasing.  Indeed, we could define the trajectory implicitly according to the differential equation $ \smfrac{d}{dt} e(t) = \xie{i}{e(t)} - e(t) $ ($ e(0) \in \swath(i) \setminus \core(i) $) that is, move from $ e(t) $ infinitesimally towards the optimal solution $ \xie{i}{e(t)} $.  Assuming this does result in a well-defined trajectory, then clearly, $ t \mapsto c^* e(t) $ is decreasing.  However, there are no clear reasons suggesting that the trajectory $ t \mapsto e(t) $ converges to $ \opt $.  It is conceivable, for example, that the trajectory reaches the boundary $ \partial \Lambdap $ in finite time, converging to a point having better objective value than $ e(0) $, but not to a point in $ \opt $.  Alternatively, in finite time the trajectory might reach $ \core(i) $.  It seems plausible that the path $ t \mapsto \xie{i}{e(t)} $ then would have limit in $ \opt $.  But how would one prove it?  How does one even rule out the possibility that in finite time, the path $ t \mapsto \xie{i}{e(t)} $ goes to infinity while the trajectory $ t \mapsto e(t) $ remains bounded but with no limit points in $ \opt $?  

Resolving these issues, and similar ones, is our primary focus.  We show that the differential equation $ \smfrac{d}{dt} e(t) = \xie{i}{e(t)} - e(t) $, $ e(0) \in \swath(i) \setminus \core(i) $, does result in well-defined trajectories in $ \swath(i) \setminus \core(i) $, and we show that either the trajectory $ t \mapsto e(t) $ or the path $ t \mapsto \xie{i}{e(t)} $ does converge to $ \opt $.  We show many other things as well, but to accurately explain, first we must formalize. 

In place of $ \smfrac{d}{dt} e(t) = \xie{i}{e(t)} - e(t) $ we often write $ \dot{e}(t) = \xie{i}{e(t)} - e(t) $.  That this dynamics results in well-defined trajectories is immediate from the following theorem, whose (relatively routine) proof is in \S\ref{s.e}.

\hypertarget{targ_thm_three}{}
\begin{thm_three}   
Assume $ 1 \leq i \leq n-2 $.
The set $ \swath(i) \setminus \core(i) $ is open in the relative topology of $ \relint(\feas) $.
Moreover, the map $ e \mapsto \xie{i}{e} $ is analytic on $ \swath(i) \setminus \core(i) $.
\end{thm_three}

As an aside, we note that every $ e \in \swath(n-2) $ has a unique optimal solution, simply due to the strict curvature of second-order cones.  Thus, we can naturally extend the definition of $ \xie{n-2}{e} $ to include all derivative directions in $ \swath(n-2) $, not just the ones in $ \swath(n-2) \setminus \core(n-2) $.  It happens that for the case $ i = n-2 $, virtually all of our results remain valid when ``$ \swath(n-2) $'' is substituted for ``$ \swath(n-2) \setminus \core(n-2) $'' (as we discuss while proving our theorems (\S\S\ref{s.c}-\ref{s.j})). With regards to Theorem \hyperlink{targ_thm_three}{3} in particular, the extended map $ e \mapsto \xie{n-2}{e} $ is analytic on all of $ \swath(n-2) $ (which is open in the relative topology of $ \relint(\feas) $).

Theorem \hyperlink{targ_thm_three}{3} implies that when initiated at $ e(0) \in \swath(i) \setminus \core(i)  $, the dynamics $ \dot{e}(t) = \xie{i}{e(t)} - e(t) $ results in a well-defined trajectory.  The trajectory remains in $ \swath(i) \setminus \core(i) $ for all time (i.e., is defined for all $ 0 \leq t < \infty $), or is defined only up to some finite time due either to reaching the boundary of $ \swath(i) \setminus \core(i) $ or escaping to infinity.  Let $ T(e(0)) $ denote the time at which the trajectory becomes undefined (possibly $ T(e(0)) = \infty $).  We refer to $ t \mapsto e(t) $ ($ 0 \leq t < T(e(0)) $) as a ``maximal trajectory.'' For brevity, we often instead write ``$ 0 \leq t < T$'' with the implicit understanding that the time of termination, $ T $, depends on $ e(0) $.  (Some of our results distinguish between $ T < \infty $ and $ T = \infty $, but never is distinction made between different finite termination times.)

Here is the formal statement of results partially described earlier:
\hypertarget{targ_main_thm_part_one}{}
\begin{main_thm_part_one} 
Assume $ 1 \leq i \leq n-2 $ and let $ t \mapsto e(t) $ $ (0 \leq t < T) $ be a maximal trajectory generated by the dynamics $ \dot{e}(t) = \xie{i}{e(t)} - e(t) $ beginning at $ e(0)  \in \swath(i) \setminus \core(i) $.
\begin{enumerate}

\item The trajectory $ t \mapsto e(t) $ is bounded, and $ t \mapsto c^* \xie{i}{e(t)}  $ is strictly increasing, with $ \val $ (the optimal value of $ \hp $) as the limit.

\item If $ T = \infty $ then every limit point of the trajectory $ t \mapsto e(t) $ lies in $ \opt $.

\item If $ T < \infty $ then the trajectory $ t \mapsto e(t) $ has a unique limit point $ \bar{e} $ and $ \bar{e} \in \core(i) $; moreover, the path $ t \mapsto \xie{i}{e(t)} $ is bounded and each of its limit points lies in $ \opt $.
\end{enumerate}
\end{main_thm_part_one}    

The Main Theorem (Part I) is proven in \S\ref{s.h}.

An immediate consequence of the theorem is that $ T = \infty $ whenever $ \core(i) = \emptyset $, as is the case whenever $ \opt \cap \Lambdaiepp{i}{e} \neq \emptyset $ for some (and hence, by Theorem~\hyperlink{targ_thm_one}{1}(A), for all) $ e \in \Lambdapp $.

Perhaps the reader wonders as to the inspiration for the idea that the trajectories $ t \mapsto e(t) $ arising from the dynamics $ \dot{e}(t) = \xie{i}{e(t)} - e(t) $ lead to optimality, either in the limits of the trajectories themselves or in the limits of the paths $ t \mapsto \xie{i}{e(t)} $.  The following theorem (whose proof is in \S\ref{s.f}) serves to clarify the inspiration, as well as to further illuminate our choice of the terminology ``{\em central} swaths'' (as opposed to simply ``swaths'').  
   
\hypertarget{targ_thm_five}{}
\begin{thm_five}  
Assume $ t \mapsto e(t) $ $ (0 \leq t < T) $ is a maximal trajectory arising from the dynamics $ \dot{e}(t) = \xie{n-2}{e(t)} - e(t) $, starting at $ e(0) \in \swath(n-2) \setminus \core(n-2) $. \begin{center}  If $ e(0) \in \cpath $ then $  \{ e(t): 0 \leq t < T \} \subseteq \cpath $. \end{center} 
\end{thm_five}

We remark that when $ i = n-2 $, the termination time $ T $ always is $ \infty $, even when $ \core(n-2) \neq \emptyset $  (see \S\ref{s.f}).

From the two theorems, we see that the central path is but one trajectory in a rich spectrum of paths.  Moreover, the central path is at the far end of the spectrum, where the cone $ \Lambdap $ is relaxed to second-order cones $ \Lambdaiep{n-2}{e} $.  Second-order cones have nice curvature properties but are far cruder approximations to $ \Lambdap $ than are cones $ \Lambdaiep{i}{e} $ for small $ i $.  This raises the interesting prospect that algorithms more efficient than interior-point methods can be devised by relying on a smaller value of $ i $, or on a range of values of $ i $ in addition to $ i = n-2 $.  

Some exploration in this vein has been made by Zinchenko (\cite{zinchenko1},\cite{zinchenko3}),  who showed for linear programs satisfying standard non-degeneracy conditions that if $ i $ is chosen appropriately and the initial derivative direction $ e(0) $ is within a certain region, then a particular algorithm based on discretizing the flow $ \dot{e}(t) = \xie{i}{e(t)} - e(t) $ converges R-quadratically to an optimal solution.

Before moving to the next result, we acknowledge that maybe the Main Theorem can be strengthened without restricting its general setting.  For example, in Part I(B) there is no statement that limit points of the path $ t \mapsto \xie{i}{e(t)} $ are optimal solutions for $ \hp $ -- there is not even a statement that the path is bounded.  This omission seems odd given that the trajectory $ t \mapsto e(t) $ is following the path $ t \mapsto \xie{i}{e(t)} $ (according to the dynamics $ \dot{e}(t) = \xie{i}{e(t)} - e(t) $) and given that the theorem states limit points of the trajectory are optimal for $ \hp $.  Intuitively, it seems the path would converge to optimality and do so even more quickly than the trajectory.  The intuition is correct for a wide variety of ``non-degenerate" problems (indeed, quicker convergence of the path than the trajectory underlies Zinchenko's speedup), but we have been unable to find a proof -- or counterexample -- in the general setting of the theorem.\footnote{To gain a sense of the difficulties (and how first impressions can mislead), consider that for any value $ 0 < T \leq \infty  $, it is straightforward to define dynamics on $ \mathbb{R}^{2} $  that generates a pair of paths $ a(t) $, $ b(t) $ for which $ \dot{b}(t) = a(t) - b(t) $ and as $ t \rightarrow T $, $ a(t) $ spirals outward to infinity whereas $ b(t) $ spirals inward to a point. (Thus, although $ a(t) $ is ``leading'' $ b(t) $, the paths end (infinitely) far apart.)}

In a similar vein, maybe it is true when $ T = \infty $ that the trajectory $ t \mapsto e(t) $ has a {\em unique} limit point.  If, like the central path, the trajectory was a semialgebraic set then the limit point indeed would be unique (simply because every semialgebraic path that has a limit point has exactly one limit point).  However, we doubt that the trajectories are necessarily semialgebraic in general, and we see no other approach to proving uniqueness.    The theorem leaves open for the general setting the possibility that when $ T = \infty $ (resp., $ T < \infty $), some trajectories $ t \mapsto e(t) $ (resp., some paths $ t \mapsto \xie{i}{e(t)} $) have non-trivial limit cycles -- yet we have no examples of such behavior.

\sep

Of course the dynamics of moving $ e $ towards an optimal solution $ x $ can also be done for $ e \in \core(i) $, in which case $ e $ would be moving towards $ x \in \optie{i}{e} = \opt $.  As a matter of formalism, it would be nice to know that such movement would result in a new derivative direction for which $ x $ is still optimal, that is, a new derivative direction that also is in $ \core(i) $.   The following theorem (proven in \S\ref{s.j}) establishes a bit more.

\hypertarget{targ_thm_six}{}
\begin{thm_six} 
Assume $ e \in \core(i) $ and let $ {\mathcal A} $ be the minimal affine space containing both $ e $ and $ \opt $.  Then
\[     {\mathcal A} \cap \Lambdapp \subseteq \core(i) \; . \]
\end{thm_six}

In the following conjecture, the empty set is taken, by default, to be convex.
\begin{conj}
$ \core(i) $ is convex. 
\end{conj}
   
\sep

Much work remains in order to transform the ideas captured in the Main Theorem (Part I) into general and efficient algorithms.  For example, devising and analyzing efficient methods for computing $ \xie{i}{e} $ given $ e $ is, in the general case, a challenging research problem.  However, computing $ \xie{n-2}{e} $ (that is, the case $ i = n-2 $)  amounts simply to solving a least-squares problem and using the quadratic formula.  Here, Chua \cite{chua}, starting with -- and extending -- ideas similar to ones above, devised and analyzed an algorithm for semidefinite programming (and, more generally, for symmetric cone programming) with complexity bounds matching the best presently known -- $ O(\sqrt{n} \log (1/\epsilon) ) $ iterations to reduce the duality gap by a factor $ \epsilon $ when $ \Lambdap $ is the cone of $ n \times n $ sdp matrices.

Although in the present work we do not analyze methods for efficiently computing $ \xie{i}{e} $, we now present a few results relevant to algorithm design.  These results also are important to the proof of the Main Theorem.

For the next two theorems, nothing is gained by distinguishing $ \ppie{i}{e} $ from any other hyperbolic polynomial, so we phrase the results simply in terms of a polynomial $ p $ hyperbolic in direction $ e $, and its first derivative $ p_e' $.  (The two theorems, moreover, do not require $ \Lambdap $ to be a regular cone.)

Let 
\[  q_e := p/p'_e \; , \]
  a rational function.  The natural domain for $ q_e $ is $ \Lambdappe{e}' $, because $ p'_e(x) > 0 $ for $ x \in \Lambdappe{e}' $ and $ p'_e(x) = 0 $ for $ x $ in the boundary of $ \Lambdappe{e}' $.

The following result is proven in \S\ref{s.g}.

\hypertarget{targ_thm_seven}{}
\begin{thm_seven}
The function $ q_e: \Lambdappe{e}' \rightarrow \reals $ is concave.
\end{thm_seven}

Previously, $ q_e $ was known to be concave on the smaller cone $ \Lambdapp $ (see \cite{bgls}).  For us, the significance of the function being concave on the larger cone $ \Lambdapp' $ is that, as the following theorem  illustrates,  $ \hp $ can be reformulated as a linearly-constrained convex optimization problem (no explicit conic constraint).    For motivation, think of the situation where one has an approximation to an optimal solution for $ \hp $ and the goal is to compute a better approximation.

\hypertarget{targ_thm_eight}{}
\begin{thm_eight}  \label{}
The optimal solutions $ x \in \Lambdappe{e}' $ for $ \hp $ are the same as for the convex optimization problem
\[ \begin{array}{rl}
\min_{x \in \Lambdappe{e}'} & - \ln c^*( e-x) - q_e(x) \\
\mathrm{s.t.} & Ax = b \; . 
\end{array} \]
\end{thm_eight}
\noindent (A caution: The theorem asserts nothing about optimal solutions for $ \hp $ that happen to lie in the intersection of the boundaries $ \partial \Lambdap $ and $  \partial \Lambdape{e}' $.)  The theorem is proven in \S\ref{s.g}.

Since, by (\ref{e.a.c}) , $ \xie{i}{e} \in \Lambdaiepp{i+1}{e} $ for $ e \in \swath(i) \setminus \core(i) $, the following corollary is immediate, except for the assertion regarding the second differential, which is established in \S\ref{s.g}.

\hypertarget{targ_cor_nine}{}
\begin{cor_nine} 
If $ 1 \leq i \leq n-2 $, $ e \in \swath(i) \setminus \core(i) $ and
\[ f(x) := - \ln c^*(e-x) - \frac{\ppie{i}{e}(x)}{\ppie{i+1}{e}(x)} \; , \]
 then $ \xie{i}{e} $ is the unique optimal solution for the convex optimization problem
\[ \begin{array}{rl}
 \min_x & f(x) \\
 \mathrm{s.t.} & Ax = b \; . 
\end{array} \]
Moreover, $ D^2f(\xie{i}{e})[v,v] > 0 $ if $ v $ is not a scalar multiple of $ \xie{i}{e} $ (in particular, if $ v \neq 0 $ satisfies $ Av = 0 $).  
\end{cor_nine}

A consequence of the assertion regarding the second differential is that Newton's method will converge quadratically to $ \xie{i}{e} $ if initiated nearby. (Which is not to say that Newton's method is the algorithm of choice for this problem.)
   
\sep

Our final result relates the dynamics $ \dot{e}(t) = \xie{i}{e(t)} - e(t) $ to the optimization problem dual to $ \hp $:
  \[  
 \left. \begin{array}{rl}
\sup_{y^*,s^*} & y^*b \\
\mathrm{s.t.} & y^* A + s^* = c^* \\
& s^* \in \Lambdap^* \end{array} \quad \right\} \, \hp^* \; ,  
\] where $ \Lambdap^* $ is the cone dual to $ \Lambdap $.\footnote{The dual cone $ \Lambdap^* $ consists of the linear functionals $ s^* $ satisfying $ s^*x \geq 0 $ for all $ x \in \Lambdap $ .}  A pair $ (y^*,s^*) $ satisfying the constraints is said to be ``strictly'' feasible if $ s^* \in \int(\Lambdap^*) $ (interior).  

Letting $ \val^* $ denote the optimal value of $ \hp^* $ ($ \val^* = - \infty $ if $ \hp^* $ is infeasible), we have, just as a matter of tracing definitions, the standard result known as ``weak duality''\footnote{This is proven simply by observing that for feasible $ x $ and $ (y^*,s^*) $,
\[  c^*x = (y^*A + s^*)x = y^*b + s^*x \geq y^*b \; , \] the inequality due to $ x \in \Lambdap $ and $ s^* \in \Lambdap^* $.}:  $ \val^* \leq \val $.  

Optimizers expect that if a dynamics provides a natural path-following framework for solving a convex optimization problem, then not only do the dynamics generate paths leading to (primal) optimality, also the dynamics somehow generate paths leading to dual optimality.  

For $ e \in \swath(i) \setminus \core(i) $, define
\[  \sie{i}{e} :=  \frac{c^*(e-x) }{\ppie{i+1}{e}(x)} D \ppie{i}{e}(x) \quad \textrm{where } x = \xie{i}{e} \; , \]
and where $ D \ppie{i}{e}(x) $ is the differential of $ \ppie{i}{e} $ at $ x $, i.e., the linear functional defined on vectors $ v $ by $ D \ppie{i}{e}(x)[v] := \smfrac{d}{dt} \ppie{i}{e}(x+tv)|_{t=0} $.
The following result is proven in \S\ref{s.i}.

\hypertarget{targ_main_thm_part_two}{}
\begin{main_thm_part_two}
Assume $ 1 \leq i \leq n-2 $  and let $ t \mapsto e(t) $ $ ( 0 \leq t < T) $ be a maximal trajectory for the dynamics $ \dot{e} = \xie{i}{e} - e $ starting at $ e(0) \in \swath(i) \setminus \core(i) $.  Then $ y^*A + \sie{i}{e(t)} = c^* $ has a unique solution $ y^* = \yie{i}{e(t)} $, and  the pair $ (\yie{i}{e(t)},\sie{i}{e(t)}) $ is strictly feasible for $ \hp^* $.  Moreover, 
\[   \yie{i}{e(t)}b  =  c^* \xie{i}{e(t)} \xrightarrow[t \rightarrow T]{} \val \] 
(in fact, increases to $ \val $ strictly monotonically) and the path $ t \mapsto (\yie{i}{e(t)},\sie{i}{e(t)}) $ is bounded.
\end{main_thm_part_two}

Consequences are, of course,  that $ \val^* = \val $ (``strong duality'') and that the limit points of $ t \mapsto (\yie{i}{e(t)},\sie{i}{e(t)}) $ form a nonempty set, each of whose elements is optimal for $ \hp^* $.

Thus, although the path $ t \mapsto \xie{i}{e(t)} $ is infeasible for $ \hp $, and the feasible trajectory $ t \mapsto e(t) $ can potentially not converge to optimality (it might instead converge to $ \bar{e}   \in \core(i) $), there is naturally generated a path $ t \mapsto (\yie{i}{e(t)},\sie{i}{e(t)}) $ that both is feasible for $ \hp^* $ and converges to optimality. 

Perhaps, then, the algorithmic framework we have posed as being for the primal optimization problem would be better posed as being for the dual, since by interchanging the primal and the dual, there would be a single path generated for the primal, and that path would be both feasible and converge to optimality.

\sep

Lastly, we mention that in defining the sequence of derivative polynomials $ \ppi{i}_e $ ($  i=1, \ldots, n-1  $), we could have used various derivative directions $ e_1, \ldots, e_{n-1} $, choosing $ e_1 $ from the hyperbolicity cone for $ p $ and defining $ p_{e_1}'(x) :=  Dp(x)[e_1] $, choosing $ e_2 $ from the (larger) hyperbolicity cone for $ p_{e_1}' $ and defining $ p_{e_1,e_2}''(x) :=  Dp_{e_1}'(x)[e_2] = D^2p(x)[e_1,e_2] $, and so on.  Several results in the following pages can be extended to this more general setting.  However, computing multidirectional derivatives can be prohibitively expensive, even for the innocuous-appearing hyperbolic polynomials $ p(x) = \prod_{i=1}^n a_i^T x $ naturally arising from polyhedral cones $ \{ x \in \reals^n: a_i^Tx \geq 0 \textrm{ for all }   i=1, \ldots, n \} $; indeed, choosing $ e_1, \ldots, e_n $ to be the standard basis, $ D^np(x)[e_1, \ldots, e_n ] $ is the permanent of the matrix whose $ i^{ \mathrm{th}} $ column is $ a_i $.  

By contrast, if the same direction $ e $ is used for all derivatives ($ i = 1, \ldots, n-1  $), the resulting polynomials $ \ppie{i}{e} $ (resp., their gradients, their Hessians) can be efficiently evaluated at any point if the initial polynomial $ p $ (resp., its gradient, its Hessian) can be efficiently evaluated at any point.  This is straightforwardly accomplished by interpolation, and can be sped up via the (inverse) Discrete Fourier Transform (see \S9 of \cite{renegar}  for some discussion).  As a  primary motivation for the present paper is designing efficient algorithms, it thus is sensible to restrict consideration to the same direction $ e $ being used for all derivatives ($ i = 1, \ldots, n-1 $).

\section{{\bf Prelude to the Analysis}}  \label{s.c} 

Now we turn to proving the results.  The theorems are proven in the order in which they were stated with the exceptions of \hyperlink{targ_main_thm_part_one}{Part I} of the Main Theorem and Theorem~\hyperlink{targ_thm_six}{6}.  The proof of the first is delayed because it depends on theorems that were stated later.  The proof of Theorem 6 is delayed, until the end, due to the combination of the proof being long and the theorem being less important than others. 

To ease burdens on the reader, each theorem is restated before its proof, and is renumbered to match the section in which it is proven (in part so it is clear there is no circularity among the proofs).  Additionally, concepts and definitions are recalled as they first retake center stage, and a few supplementary results are presented.  Thus, the reader is freed from having to refer to the preceding ``overview of results.'' 

Several proofs rely fundamentally on results from \cite{renegar}, a paper on structural aspects of hyperbolicity cones (and hyperbolic programs), a paper for which a primary goal was to provide a ready reference of ``lower-level'' details so that subsequent papers (such as the present one) could avoid drawn-out proofs.  The relevant results are presented in propositions or theorems at the beginning of sections where the results are first needed, with the exception of the following results.  

One result from \cite{renegar} is used time and time again -- the characterization we recorded as (\ref{e.a.a}), that is,
\begin{equation}  \label{e.c.a}
   \Lambdaiep{i}{e} = \{ x:\ppie{j}{e}(x) \geq  0 \textrm{ for all $j = i, \ldots, n-1 $} \} \; . 
\end{equation}
For ease of reference we record the following two characterizations that are similar to the one above:
\begin{equation}  \label{e.c.b}
   \Lambdaiepp{i}{e} = \{ x:\ppie{j}{e}(x) >  0 \textrm{ for all $j = i, \ldots, n-1 $} \}  \; , 
\end{equation} 
\begin{equation}  \label{e.c.c}
   \Lambdaiepp{i}{e} = \{ x:\ppie{i}{e}(x) >  0 \textrm{ and } \ppie{j}{e}(x) \geq 0 \textrm{ for all $j = i+1, \ldots, n-1 $} \} \; . 
\end{equation}
As was noted earlier, (\ref{e.c.a})  is immediate from Proposition 18 and Theorem 20 in \cite{renegar}.  The representations (\ref{e.c.b})  and (\ref{e.c.c})  are established in the two paragraphs following that theorem. 

Readily proven from the above characterizations are that for $ 1 \leq i \leq n-1 $,
\begin{equation}  \label{e.c.d}
 x \in \Lambdaiep{i}{e} \setminus \Lambdaiep{i-1}{e} \quad \Rightarrow \quad \ppie{i-1}{e}(x) < 0 
\end{equation}
and
\begin{equation}  \label{e.c.e}
 x \in \partial \Lambdaiep{i}{e} \quad \Rightarrow \quad \ppie{i-1}{e}(x) \leq 0 
 \; . 
\end{equation}  

On a different note, recall that the $ k^{th} $ differential $ D^k f(x) $ of a function $ f: \mathbb{R}^n \mapsto \mathbb{R}  $ is defined on $ k $-tuples of vectors by
\[      D^k f(x)[v_1, \ldots, v_k] := \smfrac{d}{dt_1} \cdots \smfrac{d}{dt_k} f(x + t_1 v_1 + \cdots t_k v_k)|_{t_k = 0} \cdots |_{t_1 = 0} \; . \]
If $ f $ is analytic -- as are hyperbolic polynomials -- the differential is well-defined for every $ k $, and is symmetric (i.e., for all permutations $ \pi $ of $ \{ 1, \ldots, k \} $, the values $ D^k f(x)[v_{ \pi(1)}, \ldots, v_{\pi(k)}] $ are identical).  Moreover, the differentials are multilinear, that is, $ v_i \mapsto D^kf(x)[v_1, \ldots, v_k] $ is linear when $ v_1, \ldots, v_{i-1}, v_{i+1}, \ldots, v_k $ are fixed.

Fixing vectors $ w_1, \ldots, w_j $ results in a symmetric, $ (k-j) $-multilinear form 
\[  [v_1, \ldots, v_{k-j}] \mapsto D^kf(x)[w_1, \ldots, w_j, v_1, \ldots, v_{k-j}] \; . \]
This form is denoted  $ D^kf(x)[w_1, \ldots, w_j] $.

A fact used extensively, and which is easily proven (by, say, induction on $ j $ with base case $ j = 0 $), occurs when $ f $ is homogeneous of degree $ \ell $ (i.e., $ f(tx) = t^{ \ell} f(x) $): For $ 0 \leq j \leq k \leq \ell   $,  
\[   D^kf(x)[\underbrace{x, \ldots,x}_{ \textrm{$ j $ times}}] = \smfrac{(\ell -k+j)!}{( \ell -k)!} D^{k-j}f(x) \; . \]
In particular, for $ 0 \leq j \leq k \leq n- i $,
\begin{equation}  \label{e.c.f}
   D^k\ppie{i}{e}(x)[\underbrace{x, \ldots,x}_{ \textrm{$ j $ times}}] = \smfrac{(n-i-k+j)!}{(n-i-k)!} D^{k-j}\ppie{i}{e}(x) \; . 
\end{equation}
A consequence of (\ref{e.c.f})  used occasionally is that if $ 0 \leq k \leq n-i $, then
\begin{equation}  \label{e.c.g}
 D^k \ppie{i}{e}(e)  = D^{k+i} p(e)[\underbrace{e, \ldots,e}_{ \textrm{$ i $ times}}]  
 =  \smfrac{(n-k)!}{(n-i-k)!} D^{k}p(e)  
\end{equation}
-- in particular, $ D^k \ppie{i}{e}(e) $ is a positive multiple of $ D^k p(e) $. 
 
Finally, we mention that the nesting of the derivative cones often is used implicitly.  Two examples of assertions made without the nesting being mentioned: ``If $ x \in \optie{i}{e} $ and $ x \in \feas $, then $ x \in \opt $.''  ``If $ x \in \Lambdaiepp{i}{e} $ then $ \ppie{i+1}{e}(x) > 0 $.''

\section{{\bf  Proofs of Theorems 1 and 2, and the Nesting of Swaths}}  \label{s.d}

The following theorem records results from \cite{renegar} that are used in this section. 

\begin{thm}  \label{t.d.a}
For any polynomial hyperbolic in direction $ e $, the following hold:
\begin{enumerate} 

\item For $ i = 1, \ldots, n-1 $, $   \Lambdaiep{i-1}{e} \cap  \partial \Lambdaiep{i}{e}=  \Lambdap \cap \partial \Lambdaiep{i}{e}  $.

\item For $ i = 0, \ldots, n-1 $, the intersection $ \Lambdap \cap \partial \Lambdaiep{i}{e} $ is independent of $ e \in \Lambdapp $. 

\item For $ i = 0, \ldots, n-2 $, the lineality space of $ \Lambdaiep{i}{e} $ consists precisely of the points in $ \Lambdaiep{i}{e} \cap  \partial \Lambdaiep{n-1}{e} $. 

\item  The cones $ \Lambdaiep{0}{e}, \Lambdaiep{1}{e}, \ldots, \Lambdaiep{n-2}{e} $ have the same lineality space (thus, if one of the cones is regular, all are regular).  

\item If $ \Lambdap $ is regular and $ 0 \leq i \leq n-2 $, then any face of $ \Lambdaiep{i}{e} $ either is a face of $ \Lambdap $ or is an extreme direction of $ \Lambdaiep{i}{e} $.

\end{enumerate} 
\end{thm}
\begin{proof} Result (A) is Proposition 16 in \cite{renegar}; (B) is established from Theorem 12 by induction (with base case $ i = 0 $) and by use of Proposition 22 (which shows that although the definition of the sets ``$ \partial^m \Lambdap $'' appearing in Theorem 12 depends on a derivative direction $ e $, the sets actually are independent of $ e \in \Lambdapp $); (C) follows from Proposition 11;  (D) is immediate from Proposition 13 by induction with base case $ i = 1 $ (alternatively, follows from (A) and (C)); (E) is Proposition 24. \end{proof}  

Here again is Theorem~\hyperlink{targ_thm_one}{1}, but renamed:
\begin{cor}  \label{t.d.b}
Assume $ \Lambdap $ is regular and $ 0 \leq i \leq n-2 $.
\begin{enumerate}

\item  The intersection $ \Lambdap \cap \partial \Lambdaiep{i}{e} $ is independent of $ e \in \Lambdapp $ \\
(thus, a face of $ \Lambdap $ which is a boundary face of $ \Lambdaiep{i}{e} $ for some $ e \in \Lambdapp $ is a boundary face for all $ e \in \Lambdapp $).  

\item If $ e \in \Lambdapp $ then any boundary face of $ \Lambdaiep{i}{e} $ either is a face of $ \Lambdap $ \\ or is a single ray  contained in $ \Lambdaiepp{i+1}{e} \setminus \Lambdaiep{i-1}{e} $. 
\end{enumerate} 
\end{cor}
\begin{proof}    
Parts (B) and (E) of Theorem~\ref{t.d.a}  mostly prove the corollary.  It remains only to show
\begin{equation}  \label{e.d.a}
    x \in ( \partial \Lambdaiep{i}{e}) \setminus \Lambdap \quad \Rightarrow \quad x \in \Lambdaiepp{i+1}{e} \setminus \Lambdaiep{i-1}{e} \; . 
\end{equation}
However, Theorem\ref{t.d.a}(A) immediately gives
\[  x \in ( \partial \Lambdaiep{i}{e}) \setminus \Lambdap \quad \Rightarrow \quad x \notin \Lambdaiep{i-1}{e} \; , \]
and upon substituting $ i+1 $ for $ i $, immediately gives
\[  x \in \Lambdaiep{i}{e} \setminus \Lambdap \quad \Rightarrow \quad x \notin \partial \Lambdaiep{i+1}{e} \; , \]
that is, gives
\[  x \in \Lambdaiep{i}{e} \setminus \Lambdap \quad \Rightarrow \quad x \in  \Lambdaiepp{i+1}{e} \; .\]
The implication (\ref{e.d.a})  is thus established, concluding the proof.
\end{proof}

{\bf Throughout the remainder of the paper, our standard assumptions apply without being made explicit in the statements of theorems, propositions, etc.}  Recall these assumptions regard the objective function $ c^* x $, the equations $ Ax = b $, and the cone $ \Lambdap $:
\begin{itemize}

\item  $ b \neq 0 $ (in particular, the origin is infeasible) 

\item   $ A $ is surjective (i.e., onto)

\item $ c^* $ is not in the image of $ A^* $ (otherwise all feasible points would be optimal) 

\item $ \Lambdap $ is a regular cone, that is, contains no subspaces other than $ \{ 0 \} $ 

\end{itemize}
Thus, by Theorem~\ref{t.d.a}(D), $ \Lambdaiep{i}{e} $ is regular for all $ 0 \leq i \leq n-2 $.

Recall the definitions of swaths and cores:
\begin{align*}
  & \swath(i)  :=  \{ e \in \Lambdapp:  Ae = b \textrm{ and }  \optie{i}{e} \neq \emptyset \} \; , \\
& \core(i)  := \{ e \in \swath(i): \optie{i}{e} = \opt \} \; . 
\end{align*} 
Recall, too, that the above corollary (aka Theorem 1) and our standard assumptions easily imply that if $ 1 \leq i \leq n-2 $ and $ e \in \swath(i) \setminus \core(i) $, then $ \optie{i}{e} $ consists of a single point, denoted $ \xie{i}{e} $.  The corollary implies, moreover, that 
\begin{equation}  \label{e.d.b}
    \xie{i}{e} \in \relint(\feasie{i+1}{e}) \setminus \feasie{i-1}{e} \; . 
\end{equation}

In \S\ref{s.b}  we noted that from the nesting 
\[  \feasie{0}{e} \subseteq \feasie{1}{e} \subseteq \cdots \subseteq \feasie{n-1}{e} \]
easily follows the nesting of cores,
\[   \core(0) \supseteq \core(1) \supseteq \cdots \supseteq \core(n-1) \; . \]  
We now establish (a bit more than) the nesting of swaths. 
\begin{prop}  \label{t.d.c}
The swaths are nested,
\[ \swath(0) \supseteq \swath(1) \supseteq \cdots \supseteq \swath(n-1)  \] 
(thus, if any swath is nonempty, so is $ \swath(0) $ -- equivalently, so is $ \opt $).  

Moreover, if $ \swath(n-1) \neq \emptyset $, or if $ \swath(i) \setminus \core(i) \neq \emptyset $ for some $ 1 \leq i \leq n-2 $, then $ \opt $ is a bounded set.
\end{prop}
\begin{proof} We know $ \core(i) \subseteq \core(i-1) $, so to prove the nesting of swaths, we assume $ e \in \swath(i) \setminus \core(i) $ and show $ e \in \swath(i-1) $.  

First consider the case $ 1 \leq i \leq n-2 $, where  $ \optie{i}{e} $ consists of the single point $ \xie{i}{e} $ -- in particular, $ \optie{i}{e} $ is nonempty and bounded, from which follows by standard convexity arguments that the level sets $ \{ x \in \feasie{i}{e}: c^*x \leq \alpha \} $ ($ \alpha \in \mathbb{R} $) are bounded.  Since $ \feasie{i-1}{e} \subseteq \feasie{i}{e} $, the level sets $ \{ x \in \feasie{i-1}{e}: c^*x \leq \alpha \} $ also are bounded.  From this and the nonemptiness of $ \feasie{i-1}{e} $ (indeed, $ e \in \feasie{i-1}{e} $) easily follows $ \optie{i-1}{e} \neq \emptyset $, that is, $ e \in \swath(i-1) $, as desired.

Observe, too, that the boundedness of the level sets $ \{ x \in \feasie{i}{e}: c^*x \leq \alpha \} $, and the relations $ \feasie{i}{e} \supseteq \feasie{j}{e}  $ for $ j = 0, \ldots, i $, imply $ \optie{j}{e} $ is bounded for all $ j = 0, \ldots, i $ -- in particular, $ \opt = \optie{0}{e} $ is bounded, thereby establishing the final statement of the proposition for the case that $ \swath(i) \setminus \core(i) \neq \emptyset  $ for some $ 1 \leq i \leq n-2 $.  (We note as an aside that always, $ \core(0) = \swath(0) $, so never is $ \core(0) $ a proper subset of $ \swath(0) $.) 

We have left to prove that $ \swath(n-1) \subseteq \swath(n-2) $, and that if $ \swath(n-1) \neq \emptyset $, then $ \opt $ is bounded.  Here we rely on the following result, which is immediate from (C) and (D) of Theorem~\ref{t.d.a} :
\begin{enumerate}
\addtocounter{enumi}{5}

\item For $ i = 0, \ldots, n-2 $, the lineality space of $ \Lambdap $ consists precisely of the points in $ ( \partial \Lambdaiep{n-1}{e}) \cap \Lambdaiep{i}{e} $. 

\end{enumerate}

Now, for every $ e \in \Lambdapp $, the cone $ \Lambdaiep{n-1}{e} $ is a halfspace, i.e., $ \Lambdaiep{n-1}{e} = \{ x: d^*x \geq 0 \} $ for some linear functional $ d^* $.   Moreover, since $ \Lambdap $ is a regular cone, the hyperplane $ \{ x: d^*x = 0 \} $  intersects the cone $ \Lambdaiep{n-2}{e} $ only at the origin (by (F) with $ i = n-2 $).   Consequently, the level sets $ \{ x \in \Lambdaiep{n-2}{e}: d^* x \leq \alpha \} $ ($ \alpha \in \mathbb{R} $) are bounded. 

Assume $ e \in \swath(n-1) $. Then, clearly, 
\[  \optie{n-1}{e} = \{ x: Ax = b \textrm{ and } d^*x = 0 \} \; , \]
and the first-order conditions are satisfied:
\[  c^* = y^*A + \lambda d^* \quad \textrm{for some $ \lambda \geq 0 $ and $ y^* $} \; . \]
Since $ c^* $ is not in the image of $ A^* $ (by assumption), it must be that $ \lambda > 0 $, from which follows that each level set $ \{ x \in \Lambdaiep{n-2}{e}: Ax = b \textrm{ and } c^*x \leq \beta \} $ is a level set $ \{ x: Ax = b \textrm{ and } d^*x \leq \alpha  \} $ for some $ \alpha $ (depending on $ \beta $), and hence is bounded (according to the conclusion of the preceding paragraph).  From this easily follows $ \optie{n-2}{e} \neq \emptyset $, that is, $ e \in \swath(n-2) $.   It also follows, of course, that for all $ 0 \leq i \leq n-2 $, the level sets $ \{ x \in \Lambdaiep{i}{e}: Ax = b \textrm{ and } c^*x \leq \beta \} $ are bounded, and hence that $ \optie{i}{e} $ is bounded -- in particular, $ \opt = \optie{0}{e} $ is bounded. \end{proof} 
 
We close this section by restating and proving Theorem \hyperlink{targ_thm_two}{2}.

\begin{thm}  \label{t.d.d}
\quad  $  \swath(n-1) = \mathrm{Central \, Path}  $
\end{thm}
\begin{proof} The central path consists precisely of the points in $ \{ x \in \Lambdapp:  Ax = b \} $ which minimize $ f_{\eta}(x) := \eta \, c^* x - \ln p(x) $ for some $ \eta > 0 $. As the functions $ \eta $ are convex, the first-order optimality conditions are sufficient as well as necessary; thus, the central path consists precisely of the points $ e $ satisfying
\begin{equation}  \label{e.d.c}
 \left.    \begin{array}{l}  
\eta c^* - \smfrac{1}{p(e)} Dp(e) = y^*A \quad \textrm{for some $ \eta > 0 $ and $ y^* $} \\
    Ae = b \\
     e \in \Lambdapp \; . \end{array}  \quad \right\}
\end{equation}

For every $ e \in \Lambdapp $, on the other hand, from the linearity of $ x \mapsto \ppie{n-1}{e}(x) $, we have $ \ppie{n-1}{e}(x) =  D\ppie{n-1}{e}(z)[x] $ for every point $ z $ -- in particular for $ z = e $ -- and thus
\[   \Lambdaiep{n-1}{e} = \{ x: D\ppie{n-1}{e}(e)[x] \geq 0 \} \; . \]
Clearly, then, $ \optie{n-1}{e} \neq \emptyset $ if and only if
\[  c^* = \lambda D\ppie{n-1}{e}(e) + w^*A \quad \textrm{for some $ \lambda \geq  0 $ and $ w^* $} \; . \]
Since $ c^* $ is not in the range of $ A^* $ (by assumption), and since $ D\ppie{n-1}{e}(e) = (n-1)! \, Dp(e) $ (by (\ref{e.c.g})), we thus have that $ e \in \swath(n-1) $ if and only if
\begin{equation}  \label{e.d.d}
  \left.   \begin{array}{l}  
  c^* = \lambda \, (n-1)! \, Dp(e) + w^*A \quad \textrm{for some $ \lambda >  0 $ and $ w^* $} \\
    Ae = b \\
     e \in \Lambdapp \; . \end{array}  \quad \right\}
\end{equation}
Obviously, the conditions (\ref{e.d.d})  and (\ref{e.d.c})  are equivalent. \end{proof}    

\section{{\bf  Proof of Theorem 3}} \label{s.e} 

Having finished proving that for each $ e \in \swath(i) \setminus \core(i) $, the set $ \optie{i}{e} $ consists of a unique point $ \xie{i}{e} $, it is time to show that the differential equation 
\[  \smfrac{d}{dt} e(t) = \xie{i}{e(t)} - e(t), \quad e(0) \in \swath(i) \setminus \core(i) \]
results in well-defined trajectories.  This is immediate from Theorem 3, which we now restate and prove.

\begin{thm}  \label{t.e.a}
Assume $ 1 \leq i \leq n-2 $.
The set $ \swath(i) \setminus \core(i) $ is open in the relative topology of $ \relint(\feas) $.
Moreover, the map $ e \mapsto \xie{i}{e} $ is analytic on $ \swath(i) \setminus \core(i) $.
\end{thm}

To establish the theorem, we introduce another theorem and a proposition.
The proof of the proposition is left to the reader, as it follows entirely standard lines, and is primarily an application of the Implicit Function Theorem to optimality conditions (see, for example, \S2.4 of \cite{fm} for similar results).
\begin{prop}  \label{t.e.b}
Assume $ f: {\mathcal E}_1 \times {\mathcal E}_2 \rightarrow \mathbb{R} $ is analytic, where $ {\mathcal E}_1 $ and $ {\mathcal E}_2 $ are Euclidean spaces.  For $ z \in {\mathcal E}_2 $, define $ f_z: {\mathcal E}_1 \rightarrow \mathbb{R} $ by $ f_z(y) := f(y,z) $, and consider for some linear functional $ y \mapsto d^*y $ the following family of optimization problems parameterized by $ z $:
\begin{equation}  \label{e.e.a}
  \begin{array}{rl}
\min_y & d^*y \\
\textrm{s.t.} & f_z(y) \geq 0 \; .  \end{array} 
\end{equation}

For some $ \bar{z} $, assume $ \bar{y} $ is a local optimum with the properties that $ Df_{\bar{z}} ( \bar{y}) \neq 0 $, and $ D^2 f_{\bar{z}}( \bar{y})[v,v] < 0 $ for all $ v \neq 0 $ satisfying $ D f_{ \bar{z}}( \bar{y})[v] = 0 $. Then   there exists an analytic function $ z \mapsto y_z $ defined on an open neighborhood of $ \bar{z} $,  and possessing the properties $ y_{ \bar{z}} = \bar{y} $ and for each $ z $ in the neighborhood, the point $ y_z $ is locally optimal for (\ref{e.e.a}).
\end{prop}

The following theorem collects results from \cite{renegar}. (Keep in mind that, as was emphasized in \S\ref{s.d}, our standard assumptions are always assumed to be in effect.  The relevant assumption below is regularity of $ \Lambdap $, but this assumption is immaterial for  parts (A) and (B).)

\begin{thm}  \label{t.e.c}
 Assume $ 1 \leq i \leq n-1 $, $ \bar{e} \in \Lambdapp $   and $ \bar{x}   \in ( \partial \Lambdaiep{i}{ \bar{e} }) \setminus \Lambdaiep{i-1}{ \bar{e}}  $ (equivalently, by Theorem~\ref{t.d.a}(A), $ \bar{x}   \in ( \partial \Lambdaiep{i}{ \bar{e} }) \setminus \Lambdap $).  
\begin{enumerate}

\item There exist open neighborhoods $ U $ of $ \bar{e}  $, and $ V $ of $ \bar{x}  $, with the property that
\[  e \in U \quad \Rightarrow \quad  \Lambdaiep{i}{e} \cap V = \{ x \in V: \ppie{i}{e}( x ) \geq 0 \} \; . \]

\item In a neighborhood of $ \bar{x}  $, $ \partial \Lambdaiep{i}{ \bar{e} } $ is a manifold, whose tangent space at $ \bar{x}  $ is $ \{ v: D\ppie{i}{\bar{e} }(\bar{x} )[v] = 0 \} $.

\item Further assume $ i \leq n-2 $. Then $ D^2 \ppie{i}{\bar{e} }(\bar{x} )[v,v] < 0 $ for all vectors $ v $ that both satisfy $ D \ppie{i}{\bar{e} }(\bar{x} )[v]= 0 $ and are not scalar multiples of $ \bar{x}  $. 

\end{enumerate}
\end{thm}
\begin{proof}  
Recall the characterization (\ref{e.c.a}), that is,
\begin{equation}  \label{e.e.b}
  \Lambdaiep{i}{e} = \{ x: \ppie{j}{e}(x) \geq 0 \textrm{ for all $ j=i, \ldots, n-1 $} \}. 
\end{equation}

  For $ i = n-1 $, result (B) is trivial  because $ x \mapsto \ppie{n-1}{e}(x) $ is a linear functional, and (A) is immediate with the additional observation that the linear functional varies continuously in $ e $.  Henceforth we assume $ 1 \leq i \leq n-2 $. 

Theorem~\ref{t.d.a}(A) gives
\[ 
 \bar{x}  \in (\partial \Lambdaiep{i}{\bar{e}  }) \setminus \Lambdap \quad \Rightarrow \quad \bar{x}   \in  \Lambdaiepp{i+1}{ \bar{e} } \; ,  
\]  
and hence, 
\begin{equation}  \label{e.e.c}
  \ppie{j}{\bar{e}}( \bar{x})  > 0 \quad  \textrm{for $ j = i+1, \ldots, n-1$} \; . 
\end{equation}
   Thus, since the polynomials vary continuously in $ e $ as well as in $ x $, there exist open neighborhoods $ U $ of $ \bar{e} $, and $ V $ of $ \bar{x} $, for which
\[   (e,x) \in U \times V \quad \Rightarrow \quad \ppie{j}{e}(x) > 0 \textrm{ for all $ j = i+1 , \ldots, n-1 $} \; . \]
This and (\ref{e.e.b})  establish statement (A) of the present theorem.

In light of (A), to establish (B) we need only show $ D \ppie{i}{ \bar{e}}( \bar{x}) \neq 0 $.  However, $ D \ppie{i}{ \bar{e}}( \bar{x})[\bar{e}] = \ppie{i+1}{ \bar{e}}( \bar{x}) > 0 $ (by (\ref{e.e.c}); hence, $  D \ppie{i}{ \bar{e}}( \bar{x}) \neq 0 $.     

Finally, (C) is a restatement of Theorem 14 in \cite{renegar}, with $ \Lambdaiep{i-1}{e} $ substituted for $ \Lambdap $ (the hypothesis of regularity is satisfied due to Theorem~\ref{t.d.a}(D)).   \end{proof}

\noindent 
{\bf  {\em Proof of Theorem~\ref{t.e.a}.}} Assume $ 1 \leq i \leq n-2 $ and $ e \in \swath(i) \setminus \core(i) $. 

We claim
\begin{equation}  \label{e.e.d}
 D \ppie{i}{e}(\xie{i}{e})[e - \xie{i}{e}] \neq 0 \; . 
\end{equation}
Indeed, as $ \xie{i}{e} \in ( \partial \Lambdaiep{i}{e}) \setminus \Lambdaiep{i-1}{e} $ (using (\ref{e.d.b}), the tangent space of $ \Lambdaiep{i}{e} $ at $ \xie{i}{e} $ is precisely $ \{ v: D\ppie{i}{e}(\xie{i}{e})[v] = 0 \} $, by Theorem~\ref{t.e.c}(B).  However, $ e $ is in the interior of the convex cone $ \Lambdaiep{i}{e} $, and hence the difference $ e - \xie{i}{e} $ cannot be in the tangent space.  The claim is thus established.

Let $ \bar{{\mathcal E}}   := \{ x: Ax = 0 \} $, a Euclidean (sub)space.  Define $ f: \bar{{\mathcal E}} \times \bar{{\mathcal E}} \rightarrow \mathbb{R} $ by $ f(y,z) = \ppie{i}{z+e}(y+e) $, and let $ f_z: \bar{{\mathcal E}} \rightarrow \mathbb{R}   $ be the function $ f_z(y) :=  f(y,z) $.  

Let $ d^* $ be the projection of $ c^* $ onto $ \bar{{\mathcal E}} $ (i.e., the linear functional on $ \bar{{\mathcal E}} $ satisfying $ d^*y = c^*y $ for all $ y \in \bar{{\mathcal E}} $).  For any $ x $ satisfying $ Ax = b $, the projection of $ D\ppie{i}{e}(x) $ onto $ \bar{{\mathcal E}} $ is $ D f_0(y) $ where $ y = x - e $.  In particular, for all $ v \in \bar{{\mathcal E}} $,
\[  D \ppie{i}{e}( \xie{i}{e})[v] = D f_0( \bar{y})[v] \quad \textrm{where $ \bar{y} := \xie{i}{e} - e$} \; . \]
Since $ e - \xie{i}{e} \in \bar{{\mathcal E}} $, (\ref{e.e.d})  thus implies $ D f_0( \bar{y}) \neq 0 $. 

Similarly, for all $ v \in \bar{{\mathcal E}} $,
\[     D^2 f_0( \bar{y})[v,v] = D^2 \ppie{i}{e}(\xie{i}{e})[v,v] \; . \]
Thus, since $ \xie{i}{e} \notin \bar{{\mathcal E}} $ (because $ b \neq 0 $, by assumption), Theorem~\ref{t.e.c}(C) implies $ D^2 f_0( \bar{y})[v,v] < 0 $ for all $ 0 \neq v \in \bar{{\mathcal E}} $ satisfying $ D f_0( \bar{y})[v] = 0 $.

In all, the hypotheses of Proposition~\ref{t.e.b}  are satisfied for $ \bar{z} = 0 $ and $ \bar{y} $.  Letting $ z \mapsto y_z $ be the analytic function whose existence is ensured by the proposition, it is readily apparent that $ y_z + e $ is locally optimal for
\[ \begin{array}{rl}
      \min_x & c^* x \\
     \textrm{s.t.} & Ax = b \\
            & \ppie{i}{z+e}(x) \geq 0 \; . \end{array} \]
Since $ z \mapsto y_z $ is continuous (it's even analytic), it thus follows from Theorem~\ref{t.e.c}(A) that for $ z $ in an open neighborhood of the origin in $ \bar{{\mathcal E}} $, the point $ y_z + e $ is locally optimal -- and hence globally optimal -- for the convex optimization problem $ \hpie{i}{z+e} $ -- that is, $ y_z + e \in \optie{i}{z+e} $, and so $ z+e \in \swath(i) $.  

However, for $ z $ in a possibly smaller open neighborhood of the origin, $ y_z + e  \notin \Lambdap $, because $ y_0 + e = \xie{i}{e} \notin \Lambdap $ and $ z \mapsto y_z $ is continuous.  Consequently, for $ z $ in this open neighborhood of the origin, $ \optie{i}{z+e} \setminus \Lambdap \neq \emptyset $, and hence, $ e \notin \core(i) $; thus, $ z \in \swath(i) \setminus \core(i) $, and clearly, $ \xie{i}{z+e} = y_z + e $.

As $ e $ was an arbitrary point in $ \swath(i) \setminus \core(i) $, the proof is complete. \hfill $ \Box $   \vspace{2mm}

In closing this section, we recall that due simply to the strict curvature of (regular) second order cones, for every $ e \in \swath(n-2) $ there is a unique optimal solution of $ \hpie{n-2}{e} $, thus naturally extending the map $ e \mapsto \xie{n-2}{e} $ to all of $ \swath(n-2) $. Additionally, every $ \bar{e} \in \Lambdapp $ and every  $ 0 \neq x  \in \partial \Lambdaiep{n-2}{\bar{e} } $ satisfies properties (A), (B) and (C) of Theorem~\ref{t.e.c}  when $ i = n-2 $  (regardless of whether $ \bar{x} $ satisfies the theorem's hypothesis $ \bar{x} \notin \Lambdap $).  Consequently, the proof of Theorem~\ref{t.e.a}  easily is made to be a proof showing that the extended map $ e \mapsto \xie{n-2}{e} $ is analytic on $ \swath(n-2) $ (and $ \swath(n-2) $ is open in the relative topology of $ \relint(\feas) $).  

\section{{\bf Proof of Theorem 5}}  \label{s.f}

We have been proving the theorems of \S\ref{s.b}  in the order they appeared, and have just finished establishing Theorem 3, that is, have just finished establishing the well-definedness of the trajectories arising from the differential equation
\[   \smfrac{d}{dt} e(t) = \xie{i}{e(t)} - e(t) ,\quad e(0) \in \swath(i) \setminus \core(i) \; . \]
To continue proving theorems in the order they appeared, we would next prove Part~I of the Main Theorem, which states that either the trajectory $ t \mapsto e(t) $ or the path $ t \mapsto \xie{i}{e(t)} $ converges to optimality (perhaps both).  However, there is groundwork to be laid before proving that result, so we jump to the one stated after it, a motivational theorem pertaining to the trajectories in the special case $ i = n-2 $.  

For any $ e $, the polynomial $ x \mapsto \ppie{n-2}{e}(x) $ is quadratic, and hence for $ e \in \Lambdapp $, $ \Lambdaiep{n-2}{e} $ is a second-order cone, regularity due to Theorem~\ref{t.d.a}(D).

Recall we say that $ t \mapsto e(t) $ $ ( 0 \leq t < T) $ is a ``maximal trajectory'' for the dynamics $ \dot{e}(t) = \xie{i}{e(t)} - e(t) $ if $ e(0) \in \swath(i) \setminus \core(i) $ and $ T $ ($ = T(e(0)) $) is the time when the trajectory either reaches the boundary of $ \swath(i) \setminus \core(i) $ or escapes to infinity.

Following is a restatement and proof of Theorem \hyperlink{targ_thm_five}{5}, whose only purposes in \S\ref{s.b}  were to motivate the choice of terminology ``{\em central} swaths'' and to give insight into the origin of the idea that in general the trajectories $ t \mapsto e(t) $ (or paths $ t \mapsto \xie{i}{e(t)} $) converge to optimality.  After completing the paper, the reader likely will not have much difficulty in making the statement of the theorem more complete.\footnote{More specifically, after understanding the paper the reader likely will not have much difficulty in showing the following, in which 
\begin{align*}
   \cpath & = \{ z( \eta ): \eta > 0 \} \; ,  \\
  \ecpath & = \{ z( \eta ): \eta \in \mathbb{R}  \}
\end{align*}
(``extended central path''), and $ z( \eta ) $ solves $ \min_x \eta \, c^* x - \ln p(x) $, s.t., $ Ax = b $, $ x \in \Lambdapp $: 
\begin{itemize}

\item $ \cpath  \subseteq  \swath(n-2) $  
\item If $ e(0) \in \ecpath \cap (\swath(n-2) \setminus \core(n-2)) $, then $ T(e(0)) = \infty $ and  
\[  \{ e(t): 0 \leq t < \infty  \} = \{ e \in \ecpath : c^* e \leq c^* e(0) \} \; . \]

\item If $ e \in \ecpath \cap \core(n-2) $ then 
\[ \core(n-2) = \ell \cap \Lambdapp = \ecpath \; , \]
  where $ \ell $ is the line containing $ e $ and the (unique) point in $ \optie{n-2}{e} = \opt $.  

\end{itemize}
 }
\begin{thm}  \label{t.f.a}
Assume $ t \mapsto e(t) $ $ (0 \leq t < T) $ is a maximal trajectory arising from the dynamics $ \dot{e}(t) = \xie{n-2}{e(t)} - e(t) $, starting at $ e(0) \in \swath(n-2) \setminus \core(n-2) $. \begin{center} If $ e(0) \in \cpath $ then $   \{ e(t): 0 \leq t < T \} \subseteq \cpath $. \end{center} 
\end{thm}
\begin{proof} The proof makes use of results known to anyone familiar with the interior-point method literature.

The central path is the set $ \{ z( \eta ): \eta > 0 \} $, where $ z( \eta ) $ solves
\[ \begin{array}{rl}
  \min_x & \eta \, c^* x - \ln p(x) \\
 \textrm{s.t.} & Ax = b \\
            & x \in \Lambdapp . \end{array} \]
In the terminology of Nesterov and Nemirovski \cite{nn}, the barrier function $ x \mapsto - \ln p(x) $ is self-concordant, a fact whose implications were first developed by G\"uler \cite{guler}.
 
The barrier function is strictly convex on its domain, $ \Lambdapp $ (``strictly'' because $ \Lambdap $ is a regular cone).  Hence, for each $ \eta > 0 $, there is at most one optimizer, $ z( \eta ) $.  It is well-known from interior-point method theory that if the optimizer $ z( \eta ) $ exists for one positive value of $ \eta $, then an optimizer $ z( \eta ) $ exists for all $ \eta > 0 $, and the path $ \eta \mapsto z( \eta ) $ is analytic (using that $ x \mapsto -\ln p(x) $ is analytic).

To prove the theorem, it suffices to show that 
 if $ z( \eta ) $ lies in $  \swath(i) \setminus \core(i) $ (which we know by Theorem~\ref{t.e.a}  to be open in the relative topology of $ \relint(\feas)$), then $ \dot{z}( \eta ) $ -- the tangent vector to the central path at $ z( \eta ) $ -- is a positive multiple of the difference $ \xie{n-2}{ z( \eta )} - z( \eta ) $. Indeed, this is sufficient to ensure that if $ z( \eta_0) \in \swath(i) \setminus \core(i) $, then as $ \eta $ increases from $ \eta_0 $, the central path will remain in $  \{ e(t): 0 \leq t < T(z( \eta_0)) \} $  -- where $ e(0) = z( \eta_0) $ -- until reaching the boundary of $ \swath(i) \setminus \core(i) $, that is, will remain in the path of the trajectory until the trajectory ends, implying $ \{ e(t): 0 \leq t < T( z( \eta_0)) \}  \subseteq \cpath $.  

Now we begin proving that $ \dot{z}( \eta ) $ is indeed a positive multiple of $ \xie{n-2}{z( \eta )} - z( \eta ) $.

Assuming the central path exists, for each $ \eta > 0 $ there exists $ y(\eta )^* $ which together with $ z( \eta ) $ satisfies the first-order optimality condition  
\begin{equation}   \label{e.f.a}
\eta c^* = \smfrac{1}{p(z(\eta ))} D  p(z( \eta )) + y( \eta )^* A \; . 
\end{equation}
It is well known that if $ A $ is surjective (one of our standard assumptions), then $ y( \eta )^* $ is unique, and $ \eta \mapsto y( \eta )^* $ is analytic (using that $ x \mapsto -\ln p(x) $ is analytic).

Differentiating in $ \eta $ provides equations satisfied by $ \dot{z}( \eta ) $, the tangent vector to the central path:
\begin{equation} 
  \label{e.f.b}
  c^* = \smfrac{1}{p(z(\eta ))} D^2  p(z( \eta )) [\dot{z}( \eta )] - \smfrac{ D   p(z( \eta ))[ \dot{z}( \eta )] }{p(z( \eta ))^2} D  p(z( \eta )) + \dot{y}( \eta )^* A \; . 
\end{equation}
It is well known, moreover, that $ c^* \dot{z}( \eta ) < 0 $.
 
Using (\ref{e.f.a})  to substitute for $ D  p(z( \eta )) $ in (\ref{e.f.b}), and relying on $ A \dot{z}( \eta ) = 0 $,  yields
\begin{equation} \label{e.f.c}
 \gamma(\eta ) c^* = D^2 p( z( \eta )) [ \dot{z}(\eta ) ] + w( \eta )^* A  
\end{equation} 
for some constant $ \gamma (\eta ) $ and vector $ w( \eta )^* $.  We aim to use this condition to show that if $ z( \eta ) \in \swath(n-2) \setminus \core(n-2) $, then $ \dot{z}( \eta ) $ is a positive multiple of $ \xie{n-2}{z( \eta )} - z( \eta ) $.  To accomplish this aim, we need an appropriate characterization of $ \xie{n-2}{ z( \eta )} $, the optimal solution to $ \hpie{n-2}{z( \eta )} $.

Now, for every $ e \in \Lambdapp $, it holds that $ \Lambdaiep{n-2}{e} \cap \partial \Lambdaiep{n-1}{e} = \{ 0 \} $ (by regularity, and parts C and D of Theorem~\ref{t.d.a}).  Thus, 
$ \feasie{n-2}{e} \cap \partial \Lambdaiep{n-1}{e} = \emptyset $ 
  (because $ b \neq 0 $) and, as a consequence, 
\begin{align}
 x \in  \feasie{n-2}{e} \quad & \Rightarrow \quad  \ppie{n-1}{e}(x) > 0 \label{e.f.d} \\
                                  & \Leftrightarrow \quad D\ppie{n-2}{e}(x)[e] > 0 \nonumber \\
                                  & \Rightarrow  D\ppie{n-2}{e}(x) \neq 0 \; .  \label{e.f.e} 
\end{align}
On the other hand, by the characterization (\ref{e.c.a}) applied to $ \Lambdaiep{n-2}{e} $,
\begin{equation}  \label{e.f.f}
   \feasie{n-2}{e} = \{ x: Ax = b, \, \ppie{n-2}{e}(x) \geq 0 \textrm{ and } \ppie{n-1}{e}(x) \geq 0 \} \; . 
\end{equation}
It follows from (\ref{e.f.d}), (\ref{e.f.e}) and (\ref{e.f.f})   that necessary and sufficient conditions for a point $ x $ to be optimal for the convex optimization problem $ \hpie{n-2}{e} $ are 
\begin{equation}  \label{e.f.g} 
\left. \begin{array}{l}
  \gamma c^* = D\ppie{n-2}{e}(x) + y^*A \quad \textrm{for some $ \gamma > 0 $ and  $ y^* $} \\
                       Ax = b \\
                      x \in \partial \Lambdaiep{n-2}{e} 
\end{array}\quad  \right\} \end{equation}

Assume $ z( \eta ) \in \swath(i) \setminus \core(i) $.  For brevity, write $ z $ for $ z( \eta ) $ and $ \dot{z} $ for $ \dot{z}( \eta ) $.  Our goal is to show $ \dot{z} $ is a positive multiple of $ \xie{n-2}{ z} - z $ -- equivalently, to show $ \xie{n-2}{z} = z + t \dot{z} $ for some $ t > 0 $ -- equivalently, to show for some $ t > 0 $ that $ x = z + t \dot{z} $ satisfies the conditions (\ref{e.f.g})  when $ e = z $.

To fix the value of $ t $, consider the condition $ x \in \partial \Lambdaiep{n-2}{z} $.  To see that there exists $ t > 0 $ for which $ x = z + t \dot{z} $ is in the boundary, begin by recalling $ c^* \dot{z}  < 0 $, and hence by (\ref{e.f.a}), $  D p(z)[ \dot{z}  ]< 0 $. However, by (\ref{e.c.g}), $ Dp(z) $ is a positive multiple of $  D \ppie{n-1}{z}(z) $ -- thus, $ D \ppie{n-1}{z}(z )[\dot{z}] < 0 \ $.
Since $ x \mapsto \ppie{n-1}{z }(x) $ is linear and $ \ppie{n-1}{z }(z) > 0 $, it follows that for some $ t> 0 $, the point $  z + t \dot{z}  $ lies in $ \partial \Lambdaiep{n-1}{z} $.  By the nesting of cones, for a (smaller) value $ t > 0 $, the point $  z + t \dot{z}  $ lies in  $ \partial \Lambdaiep{n-2}{z} $.  Consider the value of $ t $ to be fixed thusly.

Of course $ x = z + t \dot{z}  $ satisfies $ Ax = b $.  To complete the proof, it remains to show there exist $ \gamma > 0 $ and $ y^* $ satisfying
\begin{equation}  \label{e.f.h}  
  \gamma c^* = D\ppie{n-2}{z}(z + t \dot{z})  + y^*A \; . 
\end{equation}
We claim, however, there is no need to be concerned with the sign of $ \gamma $, because so long as (\ref{e.f.h})  is satisfied, $ \gamma $ is forced to be positive.  Indeed, if (\ref{e.f.h})  held with $ \gamma < 0 $, then $ z + t \dot{z}  $ would satisfy the (sufficient as well as necessary) optimality conditions for the problem obtained from $ \hpie{n-2}{z} $ by replacing ``$ \min c^*x $'' with ``$ \max c^*x $'', contradicting that $ c^*z > c^*(z + t \dot{z}) $. And if  (\ref{e.f.h})  held with $ \gamma = 0 $, then by Theorem~\ref{t.e.c}(B), $ \{ x: Ax = 0 \} $ would be a subspace of the tangent space of $ \partial \Lambdaiep{n-2}{z} $ at $ z + t \dot{z} $ -- but this would lead to a contradiction, as  $ z $ is in the interior of $  \Lambdaiep{n-2}{z} $ and $ A \dot{z} = 0 $.  Thus, we need only be concerned with showing there exist $ \gamma $ and $ y^* $ satisfying (\ref{e.f.h}), and not be concerned with the sign of $ \gamma $.

By linearity of $ x \mapsto D \ppie{n-2}{z}(x) $ (as $ \ppie{n-2}{z} $ is a quadratic polynomial),
\[ 
 D\ppie{n-2}{z}(z + t \dot{z} )  = D\ppie{n-2}{z}(z) + t D\ppie{n-2}{z}( \dot{z}) \; . \] 
Since $  D\ppie{n-2}{z}(z) $ is a positive multiple of $ Dp(z) $ (by (\ref{e.c.g})), and since $ Dp(z) $ satisfies (\ref{e.f.a}), to show that $ x = e + t \dot{z} $ satisfies (\ref{e.f.h})  for some $ \gamma $ and $ y^* $, it thus suffices to show
\begin{equation}  \label{e.f.i}
  D\ppie{n-2}{e}( \dot{z}) = \bar{ \gamma} c^* + \bar{y}^*A  \quad \textrm{for some $ \bar{\gamma}   $ and $ \bar{y}^* $} \; . 
\end{equation} 

Now, for any vectors $ u $ and $ v $,
\begin{align*}
  D^2 \ppie{n-2}{z}(u)[v] & = \left. \smfrac{d}{dt} \left( D \ppie{n-2}{z}(u + tv) \right) \right|_{t=0} \\ & = D \ppie{n-2}{z}(v) \quad \textrm{(by linearity of $ x \mapsto D \ppie{n-2}{z}(x) $)} \; . 
\end{align*}
In particular,
\[ 
 D\ppie{n-2}{z}( \dot{z}) = D^2 \ppie{n-2}{z}(z)[ \dot{z}] \; , \]
and hence, by (\ref{e.c.g}), 
\[ 
  D\ppie{n-2}{z}( \dot{z}) \textrm{ is a positive multiple of }  D^2 p(z)[ \dot{z}] \; . \] 
Consequently, that $ \bar{\gamma}   $ and $ \bar{y}^*  $ can be chosen to satisfy (\ref{e.f.h})  follows from  (\ref{e.f.c}), thus completing the proof.  \end{proof}  

\section{{\bf  Proofs of Theorems 7 and 8}}  \label{s.g}

We now consider the optimization problems $ \hpie{i}{e} $ for all $ 0 \leq i \leq n-1 $.
 The purpose of the present section is to establish a useful characterization of those optimal solutions of $ \hpie{i}{e} $ which do not lie in $  \partial \Lambdaiep{i+1}{e} $.  For $ e \in \swath(i) \setminus \core(i) $, the characterization  precisely identifies the unique optimal solution $ \xie{i}{e} $ (because $ \xie{i}{e} \notin \partial \Lambdaiep{i+1}{e} $ by (\ref{e.d.b})).  Although only the case $ e \in \swath(i) \setminus \core(i) $ is relevant to the Main Theorem, we record the characterization generally, because it has the potential for computational relevance also when $ e \in \core(i) $. 

In this section, the only properties used of the polynomials $ \ppie{i}{e} $ is that they are hyperbolic and nonconstant.  {\bf    Thus, to ease notation, for this section we let $ p $ be any nonconstant polynomial  which is hyperbolic in direction $ e $.}  We let $ p'_e $ denote the derivative polynomial.  

If $ p $ is linear, then $ p_e' \not\equiv 0 $ is a constant polynomial and its hyperbolicity cone $ \Lambdapp' $ is the entire Euclidean space $ {\mathcal E} $.

To aid the reader's intuition, the following proposition (not used in the sequel) explains  which points in $ \opt $ can possibly be not in $ \partial \Lambdape{e}' $. 
\begin{prop}  \label{t.g.a}
If $ e \in \Lambdapp $  then 
\begin{center} either \quad $ \opt \subset \partial \Lambdape{e}' $ \quad or \quad $ \opt \setminus \partial \Lambdape{e}' = \relint(\opt) \; .    $
\end{center} 
\end{prop}
\begin{proof} The proposition is trivially true when $ p $ is linear, because then, $ \partial \Lambdape{e}' = \emptyset $ and $ \relint(\opt) = \opt $ (due to $ \opt $ being an affine space).  The proposition also is trivially true if $ \opt $ is the empty set or consists of a single point.  Thus, assume $ \deg(p) \geq 2 $, assume $ \opt $ has more than one point, and assume $ \opt $ is not contained in $ \partial \Lambdape{e}' $.  

Since $ \opt \setminus \partial \Lambdape{e}' \neq \emptyset $ (by assumption), from the containment $ \opt \subset \Lambdape{e}' $, and from the convexity of the sets $ \opt $ and $ \Lambdape{e}' $, follows by standard arguments that  
\[ 
  \relint(\opt) \cap \partial \Lambdape{e}' = \emptyset \; . \]
Thus, fixing optimal $ x $ not contained in $  \relint(\opt) $, it remains to show $ x \in \partial \Lambdape{e}' $, that is, to show $ p_e'(x) = 0 $.

Choose $ y \in \relint(\opt) $, and let $ y(t) := x + t(y-x) $ ($ t \in \mathbb{R} $). Clearly, if $ y(t) \in \feas $ then $ y(t) \in \opt $.  Thus, since $ x $ is in the relative boundary of $ \opt $, it holds that $ y(t) \notin \Lambdap $ for $ t < 0 $.

Since the line segment with endpoints $ x $ and $ y $ is contained in $ \partial \Lambdap $, we have $ p(y(t)) = 0 $ for $ 0 \leq t \leq 1 $.  Since the only univariate polynomial with infinitely many roots is the polynomial that is identically zero, it follows $ p(y(t)) = 0 $ for {\em  all} $ t \in \mathbb{R} $.  Hence, from the characterizations
\[ 
    \Lambdaiep{k}{e} = \{ z: \ppie{j}{e}(z) \geq 0 \textrm{ for all $ j = k, \ldots, n-1$} \} \; , \] 
and the fact that $ y(t) \notin \Lambdap $ for $ t < 0 $, follows $ p_e'(y(0)) = 0 $, that is, $ p_e'(x) = 0 $, as desired. \end{proof} 

For $ e \in \Lambdapp $, 
\[ 
  \textrm{define } \;  q_e: \Lambdappe{e}' \rightarrow \mathbb{R} \; \textrm{ by } \; q_e(x) := p(x)/ p_e'(x) \; . \]  
As already mentioned, the purpose of this section is to develop a useful characterization of  the set $ \opt \setminus \partial \Lambdape{e}' $.  The characterization is that the points in the set are precisely the optimal solutions to the following linearly constrained optimization problem:
\begin{equation}  \label{e.g.a}
 \begin{array}{rl}
 \min_{x \in \Lambdappe{e}'} & - \ln c^*(e-x) - q_e(x) \\
\textrm{s.t.}   & Ax = b \; . 
\end{array} \end{equation}
Critical to achieving the characterization (and critical to characterization's relevance for computation) is Theorem \hyperlink{targ_thm_seven}{7}, now restated and proved.

\begin{thm}  \label{t.g.b}
The rational function $ q_e: \Lambdappe{e}' \rightarrow \mathbb{R} $ is concave.
\end{thm}
\begin{proof}  Introduce a new variable $ t $ and let $ P(x,t) := t \,  p(x) $, a polynomial that is easily seen to be hyperbolic in direction $ E = (e,1) $.  Let $ K' $ be the hyperbolicity cone for the derivative polynomial $ P'_E $ -- thus, $ K' $ is the connected component of $ S := \{ (x,t): p(x) + t p'_e(x) > 0 \} $ containing $ E $.  We claim $ K' $ is precisely the interior of the epigraph for $ -q_e: \Lambdappe{e}' \rightarrow \reals  $.  Since $ K' $, being a hyperbolicity cone, is convex, establishing the claim will establish the theorem.

If $ x \in \partial \Lambdape{e}' $ then $ p_e'(x) = 0 $ and, according to (\ref{e.c.e}), $ p(x) \leq  0 $, so for no $ t $ do we have $ (x,t) \in K' $.  Thus, now assuming  $ x \in \Lambdappe{e}' $ and $ t > -q_e(x) $, to establish the claim it suffices to show there is a path in $ S $ from $ (x,t) $ to $ E = (e,1) $ (because that will imply $ (x,t) $ and $ E $ are in the same connected component of $ S $).

Choose $ \bar{t} $ satisfying $ \bar{t} > \max \{ -q_e(sx + (1-s)e): 0 \leq s \leq 1 \} $ -- the maximum exists because the line segment connecting $ x $ to $ e $ is contained in the convex set $ \Lambdappe{e}' $ and because $ p'_e $ is positive everywhere in  $ \Lambdappe{e}' $.  It is easily verified that a path in $ S $ from $ (x,t) $ to $ (e,1) $ is obtained with three line segments; the line segment between $ (x,t) $ and $ (x, \bar{t} ) $, the line segment between $ (x, \bar{t} ) $ and $ (e, \bar{t}) $, and the line segment between $ (e, \bar{t}) $ and $ (e,1) $.    \end{proof}  

We now restate and prove Theorem \hyperlink{targ_thm_eight}{8}, the main result of this section.
\begin{thm}  \label{t.g.c}
If  $ e \in \Lambdapp $ then
\[ \opt \setminus \partial \Lambdape{e}' = \{ x: x \textrm{ is optimal for (\ref{e.g.a})} \}    \]
(possibly the empty set).
\end{thm}
\begin{proof} 
For $ x $  in $ \Lambdappe{e}' $ we have $ Dp(x) \neq 0 $ (because $ 0 < p'_e(x) = Dp(x)[e] $).
Thus, from the characterization
\[   (\partial \Lambdap) \setminus \partial \Lambdape{e}' = \{ x: p(x) = 0 \textrm{ and } \ppie{j}{e}(x) > 0 \textrm{ for all $ j = 1, \ldots, n-1 $} \} \;  \]
(by (\ref{e.c.c})), it follows that for the convex optimization problem $ \hp $,
\[ 
   \begin{array}{c}  x \in \opt \setminus \partial \Lambdape{e} ' \\\Leftrightarrow \\ 
\lambda c^* = Dp(x) + y^* A \quad \textrm{for some $ \lambda > 0 $ and $ y^* $}  \\
      A x = b \\
       x \in (\partial \Lambdap) \setminus \partial \Lambdape{e}' \; . \end{array} 
\] 
Observe that these conditions and homogeneity of $ p $ give
\[ \lambda c^*(e-x) = Dp(x)[e-x] = p_e'(x) - n \, p(x) = p_e'(x) \; , \]
that is,
\[ 
   \lambda = \frac{p_e'(x)}{c^*(e-x)} \; . 
\]
Clearly, then,
\begin{equation}  \label{e.g.b}
\left. \begin{array}{c}  x \in \opt \setminus \partial \Lambdape{e} ' \\\Leftrightarrow \\  
\smfrac{p_e'(x)}{c^*(e-x)} c^* = Dp(x) + y^* A \quad \textrm{for some $ y^* $}  \\
      A x = b \\
       x \in (\partial \Lambdap) \setminus \partial \Lambdape{e}'  \end{array} \quad \right\}
\end{equation}

On the other hand, necessary and sufficient conditions for $ x $ to solve the convex optimization problem (\ref{e.g.a})  are
\begin{equation}  \label{e.g.c}
 \left.  \begin{array}{c}
\smfrac{1}{c^*(e-x)} c^* - Dq_e(x) = w^* A  \quad \textrm{for some $ w^* $} \\
   Ax = b \\
   x \in \Lambdappe{e}'    \end{array} \quad \right\}
\end{equation} 
Observe that these conditions along with homogeneity of $ p $ and $ p_e' $ give
\begin{align*}
  1 & = \smfrac{1}{c^*(e-x)} c^*(e-x) \\ & = Dq_e(x)[e-x]  \\
    & = \smfrac{1}{p_e'(x)} \left(  Dp(x)[e-x]  - \smfrac{p(x)}{p_e'(x)} Dp_e'(x)[e-x] \right) \\
    & =  \smfrac{1}{p_e'(x)} \left( p_e'(x) - n \, p(x) - \smfrac{p(x)}{p_e'(x)} \left(   p_e''(x) - (n-1) \, p_e'(x) \right) \right) \\
  & = 1 - \frac{p(x)}{p_e'(x)} \left( 1 +  \frac{p_e''(x)}{p_e'(x)} \right)  \; . 
\end{align*}
Consequently, as $ x \in \Lambdappe{e}' $ implies $ p_e'(x) > 0 $ and $ p_e''(x) > 0 $, it must be for optimal $ x $ that $ p(x) = 0 $. Hence, in the necessary and sufficient optimality conditions (\ref{e.g.c}), the containment $ x \in \Lambdappe{e}' $ can be replaced by $ x \in ( \partial \Lambdap) \setminus \partial \Lambdape{e}' $.  Since for $ x $ satisfying $ p(x) = 0 $ we also have $ Dq_e(x) = \smfrac{1}{p_e'(x)} Dp(x) $, we see that the conditions (\ref{e.g.c})  are equivalent to
\begin{equation}  \label{e.g.d}
\left.   \begin{array}{c}
\smfrac{1}{c^*(e-x)} c^* - \smfrac{1}{p_e'(x)} Dp(x) = w^* A \quad \textrm{for some $ w^* $}  \\
   Ax = b \\
x \in (\partial \Lambdap) \setminus \partial \Lambdape{e}'    \end{array} \quad \right\}
\end{equation}

The optimality conditions (\ref{e.g.d})  and (\ref{e.g.b})  are identical. \end{proof} 

We close this section with a restatement of Corollary \hyperlink{targ_cor_nine}{9}.

\begin{cor}  \label{t.g.d}
If $ 1 \leq i \leq n-2 $, $ e \in \swath(i) \setminus \core(i) $ and
\[ f(x) := - \ln c^*(e-x) - \frac{\ppie{i}{e}(x)}{\ppie{i+1}{e}(x)} \; , \]
 then $ \xie{i}{e} $ is the unique optimal solution for the convex optimization problem
\[ \begin{array}{rl}
 \min_x & f(x) \\
 \mathrm{s.t.} & Ax = b \; . 
\end{array} \]
Moreover, $ D^2f(\xie{i}{e})[v,v] > 0 $ if $ v $ is not a scalar multiple of $ \xie{i}{e} $ (in particular, if $ v \neq 0 $ satisfies $ Av = 0 $).  
\end{cor}
\begin{proof} The first statement is immediate from Theorem~\ref{t.g.c}.  The statement regarding $ D^2 f $ is immediate from Theorem~\ref{t.e.c}(C).  \end{proof}  

As was noted in \S\ref{s.b}, the corollary's statement regarding $ D^2 f $ implies quadratic convergence of Newton's method when initiated near to $ \xie{i}{e} $.  This is not pursued in the present paper. 

\section{{\bf Proof of Part I of Main Theorem}}  \label{s.h}  

Recall that for the differential equation
\[ \smfrac{d}{dt} e(t) = \xie{i}{e(t)} - e(t) \; , \quad e(0) = e \in \swath(i) \setminus \core(i) \; , \]
we let $ T = T(e(0)) $ denote the time (possibly $ T = \infty $) at which the trajectory terminates due either to reaching the boundary of $ \swath(i) \setminus \core(i) $ or escaping to infinity.  We say that $ t \mapsto e(t) $ $ ( 0 \leq t < T ) $ is a ``maximal trajectory.''

The primary goal of this section is to prove that the path $ t \mapsto e(t) $, or the trajectory $ t \mapsto \xie{i}{e(t)} $ (perhaps both), converges to optimality  as $ t \rightarrow T $.

As we shall see, the proof is reasonably straightforward when $ T = \infty $, in which case the bounded trajectory $ t \mapsto e(t) $ has all limit points in $ \opt $.  

Perhaps deserving of mention is that for all $ e \in \swath(n-2) \setminus \core(n-2) $, it holds $ T(e) = \infty $, and so our reasonably straightforward proof applies.

Unfortunately, the case $ T < \infty $ is an entirely different matter.  Here we show $ t \mapsto e(t) $ converges to a unique point in $ \core(i) $, and we show the path $ t \mapsto \xie{i}{e(t)} $ is bounded, with all limit points in $ \opt $.  This proof is subtle and long.  It would be great if a more direct proof were discovered.

Before restating the theorem and beginning the proof, we introduce a new use of differentials that plays a significant role in this section and later.

As was recalled for the reader in \S\ref{s.c}, the $ j^{th} $ differential at $ x $ for an analytic function $ f $ is the symmetric multilinear form defined on $ j $-tuples of vectors $ [v_1, \ldots, v_j] $ by
\[  D^j f(x)[v_1, \ldots, v_j] := \smfrac{d}{dt_1} \cdots \smfrac{d}{dt_j} f(x + t_1 v_1 + \cdots t_j v_j)|_{t_j = 0} \cdots |_{t_1 = 0} \; . \]
Until now, significant roles have been played only by the differentials  $ D^j\ppie{k}{e}(x) $ of the hyperbolic polynomials $ x \mapsto \ppie{k}{e}(x) $.  Here, $ e $ is viewed as fixed, and $ x $ as the variable. 

In this section, we sometimes need to view $ e $ as a variable, not just $ x $.  We need to differentiate once with respect to $ e $, and $ j $ times with respect to $ x $.  For this we introduce the notation $ D_e ( D^j \ppie{k}{e}(x)) $ to represent the form that assigns to pairs of tuples $ [u] $ and $ [v_1, \ldots, v_j ] $ the value 
\[  \smfrac{d}{ds} (D^j \ppie{i}{e+su} (x)[v_1, \ldots, v_j])|_{s=0} \; .  \]
(Here we differentiated at $ e $ in direction $ u $ only after differentiating at $ x $ in directions $ v_1, \ldots, v_j $, but the order of differentiation is immaterial, because $ (e,x) \mapsto \ppie{i}{e}(x) $ is a polynomial (in particular, is analytic)).  Thus, there is good reason to also denote the form by, say, $ D^j( D_e \ppie{k}{e}(x)) $.)

When acting on a specific pair of tuples $ [u] $ and $ [v_1, \ldots, v_j] $, we denote the assigned value by $ D_e( D^j \ppie{k}{e}(x)[v_1, \ldots, v_j])[u] $ (or by $ D^j( D_e \ppie{k}{e}(x)[u])[v_1, \ldots, v_j] $).  The form is multilinear in $ u, v_1, \ldots, v_j $, and is symmetric in $ v_1, \ldots, v_j $.

The subscript on ``$ D_e $'' does not refer to a specific hyperbolicity direction, but rather, to the derivative being taken with respect to the hyperbolicity direction.  Thus, for example, $ D_e( D^j \ppie{k}{e(t)}(x)) $ is the form that assigns to pairs $ [u] $ and $ [v_1, \ldots, v_j] $ the value
 $   \smfrac{d}{ds} (D^j \ppie{k}{e(t)+su} (x)[v_1, \ldots, v_j])|_{s=0} $.

We have occasion to fix the direction $ u $, in which case results a form on tuples $ [v_1, \ldots, v_j] $; specifically, 
\[   [v_1, \ldots, v_j] \mapsto D_e( D^j \ppie{k}{e}(x)[v_1, \ldots, v_j])[u] \; , \]
a form we denote by $ D_e( D^j \ppie{i}{e}(x))[u] $.  Similarly, if $ v_1, \ldots, v_i $ and $ u $ are fixed, we use $ D_e(D^j \ppie{k}{e}(x)[v_1, \ldots, v_i])[u] $ to denote the form
\[  [w_1, \ldots, w_{j-i}] \mapsto D_e( D^j \ppie{k}{e}(x)[v_1, \ldots, v_i, w_1, \ldots, w_{j-i}])[u] \; . \]

Time and again we rely on a relation between the forms $ D_e(D^j \ppie{k}{e}(x)) $ and $ D^{j+1} \ppie{k-1}{e}(x) $ when $ k \geq 1 $:  
\begin{align}
   D_e(D^j\ppie{k}{e  }(x))  & =   D_e( D^{j+k} p(x)[\underbrace{e, \ldots,e}_{ \textrm{$ k $ times}}]) \nonumber \\ & = k \,  D^{j+k} p(x) [\underbrace{e, \ldots,e}_{ \textrm{$ k-1 $ times}}] \quad \textrm{(by multilinearity and symmetry)} \nonumber \\ & = k \, D^{j+1} \ppie{k-1}{e}(x)  \; .  \label{e.h.a}
\end{align}
That is, for any $ \bar{e} $ and tuples $ [u] $ and $ [v_1, \ldots, v_j] $, 
\[ D_e( D^j \ppie{k}{ \bar{e}} (x)[v_1, \ldots, v_j])[u] = k \, D^{j+1} \ppie{k-1}{ \bar{e}} (x) [u, v_1, \ldots, v_j] \]
(a consequence being that the form $ D_e( D^j \ppie{k}{ \bar{e}}(x)) $  is symmetric in $ u, v_1, \ldots, v_j $, not just in $ v_1, \ldots, v_j $).   \vspace{5mm}

Here is the restatement of the result to which this section is devoted, \hyperlink{targ_main_thm_part_one}{Part I} of the Main Theorem:

\begin{thm}  \label{t.h.a}
Assume $ 1 \leq i \leq n-2 $ and let $ t \mapsto e(t) $ $ (0 \leq t < T) $ be a maximal trajectory generated by the dynamics $ \dot{e}(t) = \xie{i}{e(t)} - e(t) $ beginning at $ e(0)  \in \swath(i) \setminus \core(i) $.
\begin{enumerate}

\item The trajectory $ t \mapsto e(t) $ is bounded, and $ t \mapsto c^* \xie{i}{e(t)}  $ is strictly increasing, with $ \val $ (the optimal value of $ \hp $) as the limit.

\item If $ T = \infty $ then every limit point of the trajectory $ t \mapsto e(t) $ lies in $ \opt $.

\item If $ T < \infty $ then the trajectory $ t \mapsto e(t) $ has a unique limit point $ \bar{e} $ and $ \bar{e} \in \core(i) $; moreover, the path $ t \mapsto \xie{i}{e(t)} $ is bounded and each of its limit points lies in $ \opt $.
\end{enumerate} 
\end{thm}

\subsection{The case $ T = \infty $} \label{s.h.a}  As mentioned, our proof of Theorem~\ref{t.h.a}  when $ T = \infty $ is much easier than when $ T < \infty $.  In this subsection, we focus on the case $ T = \infty $, although the following proposition is important also to the case $ T < \infty $.

\begin{prop}  \label{t.h.b}
Under the hypotheses of Theorem~\ref{t.h.a}, the trajectory $ t \mapsto e(t) $ is bounded, and $ t \mapsto c^* \xie{i}{e(t)} $ is strictly increasing.
\end{prop}
\begin{proof}   Since $ e(0) \in \swath(i) \setminus \core(i) $, the set $ \optie{i}{e(0)} $ is bounded (indeed, consists of the single point $ \xie{i}{e(0)} $), and hence so are the level sets $ \{ x \in \feasie{i}{e(0)}: c^* x \leq \alpha \} $ ($ \alpha \in \mathbb{R} $).  As $ \feas \subseteq \feasie{i}{e(0)} $, the level sets $ \{ x \in \feas: c^* x \leq \alpha \} $ are bounded, too.  Hence, since the trajectory $ t \mapsto e(t) $ remains in $ \feas $, and since (clearly) $ t \mapsto c^* e(t) $ is decreasing (strictly monotonically), the trajectory lies entirely within a bounded region, that is, the trajectory is bounded.

It remains to prove that $ t \mapsto c^* \xie{i}{e(t)} $ is strictly increasing.  To ease notation, let $ x(t) := \xie{i}{e(t)} $.

From $  t \mapsto \ppie{i}{e(t)}(x(t)) \equiv  0 $ we find  that
\begin{align*}
  0 & =  \smfrac{d}{dt} \ppie{i}{e(t)}(x(t)) \\
& =  D\ppie{i}{e(t)}(x(t))[ \dot{x}(t)]  + D_e(\ppie{i}{e(t)}(x(t))[ \dot{e}(t)] \\
   & =  D\ppie{i}{e(t)}(x(t))[ \dot{x}(t)] + i D \ppie{i-1}{e}(x(t)) [ \dot{e}(t)] \quad \textrm{(using (\ref{e.h.a}))} \\
   & = D\ppie{i}{e(t)}(x(t))[ \dot{x}(t)] + i \left( (n-i+1) \ppie{i-1}{e(t)}(x(t)) - \ppie{i}{e(t)}(x(t)) \right) \; ,    
\end{align*}
where the final equality is due to $ \dot{e}(t) = x(t) -e(t) $ and, by (\ref{e.c.f}),
 \[  D \ppie{i-1}{e(t)}(x(t))[x(t)] = (n-i+1) \ppie{i-1}{e(t)}(x(t)) \; . \] 
Thus, since $ \ppie{i}{e(t)}(x(t)) = 0 $ and $ \ppie{i-1}{e(t)}(x(t)) < 0 $ (by (\ref{e.c.d})), we have
\begin{equation}  \label{e.h.b}
  D\ppie{i}{e(t)}(x(t))[ \dot{x}(t)] > 0 \; . 
\end{equation}

On the other hand, $ x(t) $ satisfies the first-order condition 
\[    \lambda(t) c^* = D\ppie{i}{e(t)}(x(t)) + y(t)^* A  \qquad 
\textrm{for some $ \lambda(t) > 0 $ and $ y(t)^* $} \; .  
\] 
 Applying both sides to $ \dot{x}(t) $, using $ A \dot{x}(t) = 0 $   and substituting (\ref{e.h.b}), shows $ c^* \dot{x}(t) > 0 $. Thus, $ t \mapsto c^* x(t) $ is strictly increasing.
\end{proof} 

We are now in position to easily prove Theorem~\ref{t.h.a}  when $ T = \infty $.

\begin{thm}  \label{t.h.c}
Under the hypotheses of Theorem~\ref{t.h.a}, if $ T = \infty $ then
\begin{itemize}

\item $ t \mapsto e(t) $ is bounded, with limit points lying in $ \opt $, and

\item $ t \mapsto c^* \xie{i}{e(t)} $ increases strictly monotonically to $ \val $, the optimal value of $ \hp $.

\end{itemize}
\end{thm}
\begin{proof} 
Due to Proposition~\ref{t.h.b}, it remains only to prove that both $ t \mapsto c^* e(t) $ and $ t \mapsto c^* \xie{i}{e(t)} $ converge to $ \val $ as $ t \rightarrow \infty $.  However, the convergence of $ c^* e(t) $ to $ \val $ is immediate from 
\begin{equation}  \label{e.h.c}
T = \infty, \quad   c^* \xie{i}{e(t)} \leq \val \leq c^* e(t) \quad \textrm{and} \quad   \dot{e}(t) = \xie{i}{e(t)} - e(t)  \; . \end{equation}
Due to the monotonicity of $ t \mapsto c^* \xie{i}{e(t)} $ (by Proposition~\ref{t.h.b}), the convergence of $ c^* \xie{i}{e(t)} $ to $ \val $ also is immediate from (\ref{e.h.c})  (indeed, otherwise at some time there would occur $ c^* e(t) < \val $, contradicting $ e(t) \in \feas $).  
\end{proof}    

\subsection{The case $ T < \infty $}  \label{s.h.b}
 
For the challenging case $ T < \infty $, we split most of the analysis into two propositions.

\begin{prop}  \label{t.h.d}
Under the hypotheses of Theorem~\ref{t.h.a}, if $ T < \infty $, then letting $ x(t) := \xie{i}{e(t)} $,
either
\begin{enumerate}

\item  The trajectory $ t \mapsto e(t) $ has a unique limit point, the limit point is contained in $ \core(i) $, and the path $ t \mapsto x(t) $ is bounded, with all limit points contained in $ \opt $.

\end{enumerate}
or
\begin{enumerate}
\addtocounter{enumi}{1}

\item   It holds that $ \liminf_{t \rightarrow T} \|x(t)\| = \infty $.  Moreover, there exists a constant $ \alpha < 0 $  such that for all $ 0 \leq t < T $,
\[  \ppie{i-1}{e(t)}(x(t)) \leq \alpha \|x(t)\|^{n-i+1} \; . \]
\end{enumerate} 
\end{prop}   
\vspace{1mm}

\begin{prop}  \label{t.h.e}
Under the hypotheses of Theorem~\ref{t.h.a}, and regardless of whether $ T $ is finite, letting $ x(t) := \xie{i}{e(t)} $, for all $ 0 \leq t < T $ we have
\[  \frac{d}{dt} \,  \ln \frac{\left(  -\ppie{i-1}{e(t)}(x(t)) \right)^{n-i}}{\left(  \ppie{i+1}{e(t)}(x(t)) \right) ^{n-i+1}}   \leq 
C_1 \frac{\ppie{i-1}{e(t)}(x(t))}{\ppie{i+1}{e(t)}(x(t))} + C_2 \frac{\ppie{i-2}{e(t)}(x(t))}{\ppie{i-1}{e(t)}(x(t))} + C_3 \; , \]
where
\[  C_1 = i(n-i+1)^2 \; , \quad  C_2 = (i-1)(n-i)(n-i+2) \quad \textrm{and} \quad   C_3 = 2n-i+1 \]
  (and where for the case $ i = 1 $ we define $ \ppie{-1}{e(t)} \equiv 0 $).
\end{prop}
\vspace{1mm}

Before proving the two propositions, we show how they complete the proof of Theorem~\ref{t.h.a}.

\begin{thm}  \label{t.h.f}
Under the hypotheses of Theorem~\ref{t.h.a}, if $ T < \infty $ then
\begin{itemize}

\item $ t \mapsto e(t) $ has a unique limit point, the limit point is contained in $ \core(i) $, and
\item $ t \mapsto \xie{i}{e(t)} $ is bounded, having all limit points contained in $ \opt $; \\moreover, $ t \mapsto c^*\xie{i}{e(t)} $ increases strictly monotonically (to $ \val $, the optimal value of $ \hp $).
\end{itemize}
\end{thm}
\begin{proof}  We already know $ t \mapsto c^* \xie{i}{e(t)} $ is strictly increasing. Completing the proof thus amounts to showing case (A) of Proposition~\ref{t.h.d}  holds.  Hence, it suffices to assume that case (B) of Proposition~\ref{t.h.d}  holds, and then show Proposition~\ref{t.h.e}  yields a contradiction.  Thus, assume case (B) holds.

Since the trajectory $ t \mapsto e(t) $ is bounded (by Proposition~\ref{t.h.b}), we can choose a compact set $ K $ containing the entire trajectory (we assume nothing about $ K $ other than it is compact).  For $ j = i-2, i-1, i+1 $, define 
\[  \gamma_j := \max \{ |\ppie{j}{\bar{e} }(x)|: \bar{e}  \in K \textrm{ and } \|x\|=1 \}  \]
  (let $ \gamma_{i-2} :=  0 $ when $ i = 1 $).  Since $ \ppie{j}{\bar{e} } $ is homogeneous of degree $ n-j $, we have
\[  | \ppie{j}{\bar{e} }(x) | \leq \gamma_j \, \|x\|^{n-j} \quad \textrm{for all $ \bar{e}  \in K $ and all $ x $} \; . \]  

Let $ x(t) := \xie{i}{e(t)} $.

Using the assumed bound $ \ppie{i-1}{e(t)}(x(t)) \leq \alpha \| x(t) \|^{n-i+1} $ -- keep in mind especially that $ \alpha $ is negative -- and using the positivity of $ \ppie{i+1}{e(t)}(x(t)) $ (according to (\ref{e.d.b})), from the inequality of Proposition~\ref{t.h.e}  we find for all $ 0 \leq t < T $,
\[ \frac{d}{dt} \,  \ln \frac{\left(  -\ppie{i-1}{e(t)}(x(t)) \right)^{n-i}}{\left(  \ppie{i+1}{e(t)}(x(t)) \right) ^{n-i+1}} \leq C_1 \, \frac{ \alpha }{\gamma_{i+1}} \,  \|x(t)\|^2 + C_2 \, \frac{ \gamma_{i-2}}{| \alpha |} \, \|x(t)\| + C_3 \; .  \]
Hence, there exists a value $ r $ for which
\[ \|x(t)\| > r \quad \Rightarrow \quad  \frac{d}{dt} \,  \ln \frac{\left(  -\ppie{i-1}{e(t)}(x(t)) \right)^{n-i}}{\left(  \ppie{i+1}{e(t)}(x(t)) \right) ^{n-i+1}} < 0  \; .  \]
Since $ \liminf_{t \rightarrow T} \|x(t)\| = \infty $ (by assumption), it follows that
\[ \sup_{0 \leq t < T} \frac{\left(  -\ppie{i-1}{e(t)}(x(t)) \right)^{n-i}}{\left(  \ppie{i+1}{e(t)}(x(t)) \right) ^{n-i+1}} < \infty \; . \]
However,
\begin{align*}
    \frac{\left(  -\ppie{i-1}{e(t)}(x(t)) \right)^{n-i}}{\left(  \ppie{i+1}{e(t)}(x(t)) \right) ^{n-i+1}} & \geq \frac{\left(  | \alpha | \, \|x(t)\|^{n-i+1} \right)^{n-i}}{\left(  \gamma_{i+1} \|x(t) \|^{n-i-1} \right) ^{n-i+1}}\\
& \qquad   =   \frac{| \alpha |^{n-i}}{\gamma_{i+1}^{n-i-1}} \,  \|x(t)\|^{n-i+1} \rightarrow  \infty \; , 
\end{align*} 
a contradiction, thus concluding the proof of the theorem (except for proving the two propositions).  \end{proof}

    \subsubsection{{\bf Proving the first of the two propositions}}  Now we begin proving Proposition~\ref{t.h.d}.  The proof relies on three lemmas.

\begin{lemma}  \label{t.h.g}
Assume $ \opt $ is a bounded set and $ 1 \leq i \leq n - 2 $.  Let $ \{ e_j \} $ be a sequence of derivative directions converging to  $ e \in \Lambdapp $, and assume $ x_j \in \optie{i}{e_j} $.  
\begin{enumerate}
\item   If $ x_j \rightarrow x $, then $ x \in \optie{i}{e} $. 
\item   If $ \{ x_j \} $ is an unbounded set, then $  \optie{i}{e} = \emptyset $ and $ \liminf \| x_j \| = \infty $. 
\item   If $ \{ x_j \} $ is an unbounded set and $ \limsup c^* x_j  > - \infty  $, then   $  x_j/\|x_j\|   $ has exactly one limit point $ d $; moreover, $ d $ satisfies 
$  \ppie{i-1}{e}(d) < 0 $.  
\end{enumerate} 
\end{lemma}
\begin{proof}  First assume $ x_j \rightarrow x $.  Then $ \ppi{k}_e(x) = \lim \ppi{k}_{e_j}(x_j) $ for all $ k $.  However, by (\ref{e.c.a}), $ \ppi{k}_{e_j}(x_j) \geq 0 $ for all $ k = i, \ldots, n-1 $.  Thus, $ \ppi{k}_e(x) \geq 0 $ for all $ k = i, \ldots, n-1 $.  Hence, again invoking (\ref{e.c.a}), $ x \in \Lambdaiep{i}{e} $.  Since, trivially, $ Ax = b $, we thus see that $ x \in \feasie{i}{e} $.

To prove that $ x $ not only is feasible but is optimal, we assume otherwise and obtain a contradiction.  Thus, assume $ \bar{x} \in \feasie{i}{e} $ satisfies $ c^* \bar{x}  < c^* x  $.  By nudging $ \bar{x} $ towards the strictly feasible point $ e $, we may assume $ \bar{x} $ is strictly feasible in addition to satisfying $ c^* \bar{x}  < c^* x  $.  However, by (\ref{e.c.b}), strict feasibility of $ \bar{x} $ implies $ \ppi{k}_e( \bar{x}) > 0 $ for all $ k = i, \ldots, n-1 $.  Since $ \ppi{k}_{e_j}( \bar{x}) \rightarrow  \ppi{k}_{e}( \bar{x})  $, we thus have, again using (\ref{e.c.a}), that $ \bar{x} \in \feasie{i}{e_j} $ for $ j > J $ (some $ J $).  But then 
\[ c^*x > c^* \bar{x}   \geq  \lim c^* x_j = c^*x \; , \]
a contradiction.  Hence, $ x $ is optimal, and assertion (A) of the lemma is established.

Now assume $ \{ x_j \} $ is unbounded.  Choose a subsequence $ \{ x_{j_{ \ell} } \} $ for which \newline $ \liminf_{ \ell \rightarrow \infty } \|x_{j_{ \ell} } \|=\infty$ and let $ d $ be a limit point of $ \{ x_{j_{ \ell} }/ \| x_{j_{ \ell} } \| \} $.  Then $ d $ is a feasible direction for $ \hpie{i}{e} $, that is, satisfies $ Ad = 0 $ and $ d \in \Lambdaiep{i}{e} $. (That $ Ad = 0 $ is trivial.  That $ d \in \Lambdaiep{i}{e} $ follows from $ \ppie{k}{e}(d) = \lim_{ \ell  \rightarrow \infty} \ppie{k}{e_{j_{\ell}}}(x_{j_{ \ell }}/ \| x_{j_{ \ell }} \| ) \geq 0 $ for $ k = i, \ldots n-1 $, the inequality being due to $ x_{j_{ \ell}} \in \Lambdaiep{i}{e_{j_{ \ell}}} $ along with homogeneity of $ \ppie{i}{e_{j_{\ell}}} $.)    The feasible direction $ d $ satisfies $ c^* d \leq 0 $ (because $ e \in \feasie{i}{e_{j_{ \ell}} } $ -- and hence $ \limsup_{ \ell \rightarrow \infty}  c^* x_{j_{ \ell}} < \infty $   -- and because  $ \liminf_{ \ell \rightarrow \infty}  \| x_{j_{ \ell}} \| = \infty $). 
 
Clearly, now, if $ \optie{i}{e} \neq \emptyset $ then $ \optie{i}{e} $ is unbounded (in direction $ d $), implying by Corollary~\ref{t.d.b}(B) that $ \optie{i}{e} = \opt $.  But this would contradict our assumption that $ \opt $ is bounded.  Thus, $ \optie{i}{e} = \emptyset $.

To conclude the proof of assertion (B) of the lemma, it remains to show $ \liminf \| x_j \| = \infty $.  But if this identity did not hold then $ \{ x_j \} $ would have a limit point, and thus, by assertion (A) of the lemma, we would have $ \optie{i}{e} \neq \emptyset $, contradicting what we just proved.  The proof of assertion (B) of the lemma is now complete.

Now we prove assertion (C).  Assume $ \{ x_j \} $ is unbounded and $ \limsup  c^* x_j  > - \infty $.  

With $ \{ x_j \} $ unbounded, we already know from above that every limit point $ d $ of $ \{ x_j/\|x_j\| \} $ is a feasible direction for $ \hpie{i}{e} $ (that is, satisfies $ Ad = 0 $ and $ d \in \Lambdaiep{i}{e} $) and $ c^* d \leq 0 $.

We claim, however, that every feasible direction $ d $ for $ \hpie{i}{e} $ satisfies $ c^* d  \geq 0 $.  Indeed, otherwise the optimal objective value of $ \hpie{i}{e} $ would be unbounded and hence there would exist $ \bar{x} \in \relint(\feasie{i}{e}) $ satisfying $ c^* \bar{x}  < \limsup c^* x_j  $ (using the assumption $ \limsup  c^* x_j  > - \infty $).  But $ \bar{x} \in \feasie{i}{e_j} $ for sufficiently large $ j $ (because $ 0 < \ppie{k}{e}(\bar{x}) = \lim \ppie{k}{e_j}( \bar{x}) $ for $ k = i, \ldots, n-1 $), contradicting that  $ x_j $ is optimal for $ \hpie{i}{e_j} $ for all $ j $.

Clearly, from the two preceding paragraphs, every limit point of $ \{ x_j/\|x_j\| \} $ is contained in  $ D := \{ d \in \Lambdaiep{i}{e}: Ad = 0 \textrm{ and } c^* d  = 0 \} $.  To show  $ \{ x_j/\|x_j\| \} $ has a unique limit point it thus suffices to show  $ D $ consists of a single ray.

We claim $ D \cap \Lambdaiepp{i}{e} = \emptyset $.  Otherwise there would exist $ d \in D $ satisfying $ \ppie{k}{e}(d) > 0 $ for $ k = i, \ldots, n-1 $, which from continuity would imply $ d $ to be a feasible direction for $ \hpie{i}{e_j} $ for all $ j > J $ (some $ J $).  But then $ \optie{i}{e_j} $ would be an unbounded set for $ j > J $, implying by Corollary~\ref{t.d.b}(B) that $ \optie{i}{e_j} = \opt $, and thus contradicting the assumed boundedness of $ \opt $.  Hence, it is indeed the case that $ D \cap \Lambdaiepp{i}{e} = \emptyset $.

It is clear now that $ D = \{ d \in \partial \Lambdaiep{i}{e}: Ad = 0 \textrm{ and } c^* d  = 0 \} $. Thus, $ D $  is contained in a boundary face of $ \Lambdaiep{i}{e} $.  Hence, $ D $ consists of a single ray or is a subset of $ \Lambdap $, by Corollary~\ref{t.d.b}(B).  But if $ D $ were a subset of $ \Lambdap $ then each $ d \in  D $ would be a feasible direction for $ \hp $, which, with $ c^* d  = 0 $, would imply $ \opt $ to be unbounded, a contradiction to the assumed boundedness of $ \opt $.  Thus, $ D $ consists of a single ray, and so $ \{ x_j/\|x_j\| \} $ has a unique limit point $ d $ (and $ d $ satisfies $ d \in \partial \Lambdaiep{i}{e} $, $ Ad = 0 $, $ c^*d = 0 $).  

Lastly, we prove $ \ppie{i-1}{e}(d) < 0 $.  

As $ x_j \in \partial \Lambdaiep{i}{e_j} $, we have $ \ppie{i-1}{e_j}(x_j) \leq 0 $ (by (\ref{e.c.e})).  Thus, by homogeneity and continuity, $ \ppie{i-1}{e}(d) \leq  0 $.  However, since $ d \in \partial \Lambdaiep{i}{e} $, if $ d $ satisfied $ \ppie{i-1}{e}(d) = 0 $, then we would have $ d \in \Lambdaiep{i-1}{e} $, and hence $ d \in \Lambdap $ (by Theorem~\ref{t.d.a}(A)) -- so $ d $ would be a feasible direction for $ \hp $, and hence $ \opt $ would be unbounded (using $ c^* d = 0 $), a contradiction.  Thus, $ \ppie{i-1}{e}(d) < 0 $, completing the proof of the lemma.  \end{proof}

\begin{lemma}  \label{t.h.h}
Assume $ \opt $ is a bounded set, $ e \in \Lambdapp $ and $ 1 \leq i \leq n-2 $.  Assume there exist sequences $  e_j  \in  \swath(i) $ and $ x_j \in \optie{i}{e_j} $ for which $ e_j \rightarrow e $, $ \{ x_j \} $ is unbounded and $ \limsup c^*x_j > -\infty $. 
\begin{enumerate}

\item  There exists a constant $ \alpha < 0 $ and an open neighborhood $ U $ of $ e $ such that for all $ \bar{e} \in U \cap \swath(i)  $ and $ \bar{x} \in \optie{i}{\bar{e}} $, 
\[  \ppie{i-1}{\bar{e}}(\bar{x})  < \alpha \| \bar{x} \|^{n-i+1}   . \]

\item   For each $ \ell \in \mathbb{R}  $, there exists an open neighborhood $ V( \ell ) $ of $ e $  such that for all $ \bar{e} \in V( \ell ) \cap \swath(i) $ and $ \bar{x} \in \optie{i}{ \bar{e}} $, it holds that $ \| \bar{x} \| \geq  \ell  $.
\end{enumerate}      
\end{lemma}
\begin{proof}  Parts (A) and (B) of Lemma~\ref{t.h.g}  together imply that if for one sequence $ \{ e_j \} \subset \swath(i) $  there exist $ x_j \in \optie{i}{e_j} $ for which $ \{ x_j \} $  is unbounded -- where $ e_j \rightarrow e $ -- then for every  sequence $ \{ e_j \} \subset \swath(i) $ converging to $ e $, and every choice of $ x_j \in \optie{i}{e_j} $, we have $ \liminf   \|x_j\| = \infty $. Part (B) of the present lemma easily follows.

Towards proving part (A), let $ e_j $ and $ x_j $ be sequences as in the statement of the present lemma.  Part (C) of Lemma~\ref{t.h.g}  then shows $ x_j/\|x_j\| $ has a unique limit point $ d $.  

We claim $ \bar{x}_j/\|\bar{x}_j\| \rightarrow d $ for any sequences $ \bar{e}_j $ and $ \bar{x}_j \in \optie{i}{\bar{e}_j} $ satisfying $ \bar{e}_j \rightarrow e $.  Indeed, the intermingled sequences $ e_1, \bar{e}_1, e_2, \bar{e}_2, \ldots  $ and  $ x_1, \bar{x}_1, x_2, \bar{x}_2, \ldots  $ clearly satisfy the hypotheses of part (C) of Lemma~\ref{t.h.g}, and hence the normalized sequence
\[  x_1/\|x_1\|, \bar{x}_1/\| \bar{x}_1\|,  x_2/\|x_2\|, \bar{x}_2/ \| \bar{x}_2\|, \ldots  \]
  has a unique limit point, which, of course, is $ d $, thus establishing the claim.

Choosing, say, $ \alpha = \smfrac{1}{2} \, \ppie{i-1}{e}(d) $, part (A) of the present lemma is now a consequence of  the continuity of the polynomial $ (\bar{e}, \bar{x}) \mapsto \ppie{i-1}{\bar{e}}(\bar{x} ) $,  and the fact that $  x \mapsto \ppie{i-1}{ \bar{e}}(x)  $  is homogeneous of degree $ n-i+1 $  for all $ \bar{e} $.   \end{proof}

Throughout the remainder of the section,
\begin{itemize}
\item assume $ 1 \leq i \leq n- 2 $, 
\item  assume $ t \mapsto e(t) $ ($ 0 \leq t < T $) is a maximal trajectory defined by the differential equation $ \dot{e}(t) = \xie{i}{e(t)} - e(t) $, starting at $ e(0) \in \swath(i) \setminus \core(i) $,
\item assume $ T < \infty $,  and 
\item  to ease notation, let $ x(t) := \xie{i}{e(t)} $.
\end{itemize}

The third (and final) lemma for proving Proposition~\ref{t.h.d}  shows that when the trajectory $ t \mapsto e(t) $ terminates at finite time $ T $, the termination is not due to having reached the boundary of $ \feas $. 

\begin{lemma}  \label{t.h.i}
 \quad  $ \{ e(t): 0 \leq t < T \} \subset \Lambdapp $ 
\end{lemma}
\begin{proof}  It suffices to show $ p(e(t)) \geq p(e(0)) e^{-nt} $ for all $ t $, and hence suffices to show $ \smfrac{d}{dt} p(e(t)) \geq -n p(e(t)) $, that is, suffices to show $ Dp(e(t))[ \dot{e}(t)] \geq -n p(e(t)) $.  Since $ \dot{e}(t) = x(t) - e(t) $  and since $ Dp(e(t))[e(t)] = n p(e(t)) $ (by (\ref{e.c.f})), it suffices to show $ Dp(e(t))[x(t)] \geq 0 $. 

However, by (\ref{e.c.g}), for all $ \bar{e}   $, $ Dp(\bar{e}  ) $ is a positive multiple of $ D \ppie{n-1}{\bar{e}  }(\bar{e}  ) $.  Since for all $ x $, $ D\ppie{n-1}{\bar{e}  }(\bar{e}  )[x] = \ppie{n-1}{\bar{e}  }(x) $ (because $ x \mapsto \ppie{n-1}{\bar{e}  }(x) $ is linear), we have $ Dp(\bar{e}  )[x] \geq 0 $ if and only if $ x $ is in the cone $  \Lambdaiep{n-1}{\bar{e}  } $.  But $ x(t) \in  \Lambdaiep{i}{e(t)} \subseteq \Lambdaiep{n-1}{e(t)} $, thus concluding the proof.  
\end{proof}  

\noindent 
{\bf {\em Proof of Proposition~\ref{t.h.d}.}}  To rely on Lemmas \ref{t.h.g}  and \ref{t.h.h}, we need $ \opt $ to be a bounded set.  Boundedness of $ \opt $, however, was established in Proposition~\ref{t.h.b}  as a simple consequence of $ \swath(i) \setminus \core(i) $ being nonempty.

Let $ {\mathcal L} $ be the set of limit points for the trajectory $ t \mapsto e(t) $ as $ t \rightarrow T $.  By Lemma~\ref{t.h.i}, $ {\mathcal L} \subset \Lambdapp $.  Thus,  if $ e  \in  {\mathcal L} $, then either  $ e  \in \core(i) $ (that is, $ \optie{i}{ e } = \opt $) or $ e  \in \Lambdapp \setminus  \swath(i) $ (in which case $ \optie{i}{e } = \emptyset $).

Assume first that $ {\mathcal L} $ consists of a single point $ e $.  

Then, by Lemma~\ref{t.h.g}(A), each limit point of the path $ t \mapsto x(t) $ lies in $ \optie{i}{e} $.  If the path is bounded, the set of limit points is nonempty, hence $ \optie{i}{e} \neq \emptyset $, and thus, from above, $ e \in \core(i) $.  Clearly, then, if the path $ t \mapsto x(t) $ is bounded, case (A) of the proposition holds.

On the other hand, if the path $ t \mapsto x(t) $ is unbounded, then  Lemma~\ref{t.h.h}  implies $ \liminf_{t \rightarrow T} \|x(t) \| = \infty $,  and implies there exists $ 0 \leq \bar{t} < T $ and $ \alpha < 0 $ satisfying
\begin{equation}  \label{e.h.d}
  \ppie{i-1}{e(t)}(x(t)) \leq \alpha \|x(t)\|^{n-i+1} \quad \textrm{for all $ \bar{t} \leq t < T $}  \end{equation} 
(here we make use of the fact that $ t \mapsto c^* x(t) $ is strictly increasing (Proposition~\ref{t.h.b}), so that the hypothesis $ \limsup c^* x_j > -\infty  $ of the lemma is clearly fulfilled). 
Hence, from the compactness of the closed interval $ [0, \bar{t}] $ and the fact that for all $ t $, both $ x(t) \neq 0 $ (because $ b \neq 0 $) and $ \ppie{i-1}{e(t)}(x(t)) < 0 $ (by (\ref{e.c.d})), a perhaps larger (but still negative) value of $ \alpha $ satisfies
\begin{equation}  \label{e.h.e}
  \ppie{i-1}{e(t)}(x(t)) \leq \alpha \|x(t)\|^{n-i+1} \quad \textrm{for all $ 0 \leq t < T $} \; .  \end{equation} 
Hence, in all, if the path $ t \mapsto x(t) $ is unbounded, then case (B) of the proposition holds, concluding consideration of the case that $ {\mathcal L} $ consists of a single point.

For the remainder of the proof assume $ {\mathcal L} $ contains more than one point.  We show case (B) of the proposition holds.  We claim that for this it suffices to show  $ {\mathcal L} $ is compact, and that each $ e \in {\mathcal L} $ satisfies the hypotheses of Lemma~\ref{t.h.h}.  Indeed, compactness and Lemma~\ref{t.h.h}(A) imply  $ {\mathcal L} $ can be covered by finitely many open sets $ U_j $ for which there exist $ \alpha_j < 0 $ with the property that for all $ \bar{e} \in U_j \cap \swath(i) $ and $ \bar{x} \in \optie{i}{\bar{e}} $, it holds $  \ppie{i-1}{\bar{e}}( \bar{x} ) < \alpha_j \| \bar{x} \|^{n-i+1} $.  Consequently, letting $ \alpha := \max_j \alpha_j $, there exists $ 0 \leq \bar{t} < T $ for which (\ref{e.h.d})  holds.  Then, as in the preceding paragraph, for a possibly larger (but still negative) value of $ \alpha $, (\ref{e.h.e})  holds.  To conclude proving the claim, it remains only to show $ \liminf_{t \rightarrow T} \|x(t)\| = \infty $.  This, however, is easily accomplished  by covering $ {\mathcal L}   $, for each $ \ell \in \mathbb{R} $, by finitely many open sets $ V( \ell)_j $ as appear in Lemma~\ref{t.h.h}(B).

Thus, to complete the proof of Proposition~\ref{t.h.d}, it remains only to show $ {\mathcal L} $ is compact, and that each $ e \in {\mathcal L} $ satisfies the hypotheses of Lemma~\ref{t.h.h}. 

That $ {\mathcal L} $ is compact is trivial -- indeed, the trajectory $ t \mapsto e(t) $ is bounded, by Proposition~\ref{t.h.b}.

Since $ {\mathcal L} $ consists of more than one point, for each $ e \in {\mathcal L} $ and each open neighborhood $ W $ of $ e $, the Euclidean arc length of $  \{ e(t): 0 \leq t < T \} \cap W $ is infinite.  Thus, since $ T < \infty $,  for each $ e \in {\mathcal L} $ there exists an sequence $ t_1 < t_2 < \ldots $ satisfying $ e(t_j) \rightarrow e $ and $  \|\dot{e}(t_j)\| \rightarrow \infty $.  But $ \dot{e}(t_j) = x(t_j) - e(t_j) $, so $ \|x(t_j) \| \rightarrow \infty $.  Since, additionally, $ \limsup c^* x(t_j) > - \infty $ (because $ t \mapsto c^* x(t) $ is strictly increasing), $ e $ thus clearly satisfies the hypotheses of Lemma~\ref{t.h.h}.  \hfill $ \Box $ 

        \subsubsection{{\bf   Proving the second of the two propositions}}  In proving Proposition~\ref{t.h.e}, continue to
\begin{itemize}
\item assume $ 1 \leq i \leq n- 2 $, 
\item  assume $ t \mapsto e(t) $ ($ 0 \leq t < T $) is a maximal trajectory defined by the differential equation $ \dot{e}(t) = \xie{i}{e(t)} - e(t) $, starting at $ e(0) \in \swath(i) \setminus \core(i) $,
\item assume $ T < \infty $,  and 
\item  to ease notation, let $ x(t) := \xie{i}{e(t)} $.
\end{itemize}
Also, for every $ e $, define $ q_e := \ppie{i}{e}/\ppie{i+1}{e} $.

The proof of Proposition~\ref{t.h.e}  depends on two lemmas.  For understanding the first lemma, recall
the identity given as (\ref{e.h.a}), that is,
\begin{equation}  \label{e.h.f}
  D_e(D^j\ppie{k}{e}(x)) = k D^{j+1} \ppie{k-1}{e}(x) \; . 
\end{equation}

\begin{lemma} \label{t.h.j}
For any $ e $ and  $ x $  satisfying $ \ppie{i}{e}(x) = 0 \neq \ppie{i+1}{e}(x) $, we have
\begin{align*}
   D_e( D q_e(x))[ x-e] = & \, 
  D q_e(x) \\ & \quad + \frac{i(n-i)}{\ppie{i+1}{e}(x)} D \ppie{i-1}{e}(x)  \\ & \qquad - \frac{i(n-i+1) \ppie{i-1}{e}(x)}{(\ppie{i+1}{e}(x))^2} D\ppie{i+1}{e}(x) \; . \end{align*} 
\end{lemma}
\begin{proof}  For an arbitrary vector $ \Delta e $ and for any $ x $ at which $ q_e $ is defined (i.e., any $ x $ satisfying $ \ppie{i+1}{e}(x) \neq 0 $), use of (\ref{e.h.f})  gives
 \begin{align*}
 D_e (Dq_e(x))[ \Delta e]  = & \, D_e \left( \smfrac{1}{\ppie{i+1}{e}(x)}  D\ppie{i}{e}(x) - \smfrac{\ppie{i}{e}(x)}{\ppie{i+1}{e}(x)^2}  D \ppie{i+1}{e}(x) \right) [ \Delta e]  \\
    = & \, \smfrac{i}{\ppie{i+1}{e}(x)} D^2 \ppie{i-1}{e}(x)[ \Delta e ] \\
& \quad - \smfrac{i+1}{(\ppie{i+1}{e}(x))^2}  (D\ppie{i}{e}(x)[ \Delta e]) D\ppie{i}{e}(x) \\
& \qquad - \smfrac{i}{(\ppie{i+1}{e}(x))^2} (D\ppie{i-1}{e}(x)[\Delta e]) D\ppie{i+1}{e}(x)\\
& \qquad \quad - \smfrac{2(i+1)\ppie{i}{e}(x)}{(\ppie{i+1}{e}(x))^3} ( D\ppie{i}{e}(x)[ \Delta e]) D\ppie{i+1}{e}(x)\\
& \qquad \qquad  - \smfrac{(i+1)\ppie{i}{e}(x)}{(\ppie{i+1}{e}(x))^2}  D^2\ppie{i}{e}(x)[\Delta e ] \; . 
\end{align*}
Substitute $ \Delta e = x - e $, then use obvious identities to get rid of ``$ [e] $'' (e.g., $ D\ppie{i}{e}(x)[e] = \ppie{i+1}{e}(x) $, $ D^2 \ppie{i-1}{e}(x)[ e ] = D \ppie{i}{e}(x) $), and use homogeneity to get rid of ``$ [x] $'' -- specifically, use (\ref{e.c.f}).  Finally, substitute $ \ppie{i}{e}(x) = 0 $,  and $ \frac{1}{\ppie{i+1}{e}(x)} D\ppie{i}{e}(x) = Dq_e(x) $  (because $ \ppie{i}{e}(x) = 0 $), thereby concluding the proof.  \end{proof}

\begin{lemma}  \label{t.h.k}
For all $ 0 \leq t < T $, there exists $ y(t)^* $ satisfying
\begin{gather*} 
   D^2q_{e(t)}(x(t))[ \dot{x}(t)]  -   \frac{c^* \dot{x}(t) }{(c^*(e(t)-x(t)))^2} \,  c^* \qquad \qquad \qquad \qquad \qquad \qquad \qquad \\
\qquad \qquad   = \frac{i \ppie{i-1}{e(t)}(x(t))}{\ppie{i+1}{e(t)}(x(t))} \left( (n-i+1)\frac{ D\ppie{i+1}{e(t)}(x(t))}{\ppie{i+1}{e(t)}(x(t))}  - (n-i)\frac{ D\ppie{i-1}{e(t)}(x(t))}{\ppie{i-1}{e(t)}(x(t))} \right)   + y(t)^* A \; . 
\end{gather*}  
 \end{lemma}
\begin{proof}  Corollary~\ref{t.g.d}  shows $ x(t) $ is optimal for the convex optimization problem
\[ \begin{array}{rl}
 \min_x & - \ln c^*(e(t)-x) - q_{e(t)}(x) \\
 \mathrm{s.t.} & Ax = b  
\end{array} \]
and hence there exists $ w(t)^* $ satisfying the first-order condition
\begin{equation}  \label{e.h.g}
   \frac{1}{c^*(e(t)-x(t))} \, c^* - Dq_{e(t)}(x(t)) = w(t)^* A \; . 
\end{equation} 
Since $ t \mapsto e(t) $ and $ t \mapsto \xie{i}{e(t)} $ are analytic, so is $ t \mapsto w^*(t) $ (using that $ A $ is surjective, by assumption).  Differentiating in $ t $ gives
\[\frac{c^*( \dot{e}(t)- \dot{x}(t))  }{(c^*(e(t)-x(t)))^2} \,  c^* + D^2q_{e(t)}(x(t))[ \dot{x}(t) ] + D_{e(t)}(Dq_{e(t)}(x(t)))[ \dot{e}(t)] =  - \dot{w}(t)^* A  \; . \]
To complete the proof, substitute 
\[  \dot{e}(t) = x(t) - e(t) \quad \textrm{and} \quad   \frac{1}{c^*(e(t)-x(t)) } c^* = Dq_{e(t)}(x(t)) + w(t)^*A \; \, \, \textrm{(by (\ref{e.h.g}))} \; , \]
and  then use Lemma~\ref{t.h.j}  to substitute for $ D_{e(t)}(Dq_{e(t)}(x(t)))[ x(t) - e(t)] $. \end{proof} 

\noindent 
{\bf {\em Proof of Proposition~\ref{t.h.e}.}}  To temper notation, for arbitrary $ t $ satisfying $ 0 \leq t < T $, let 
\[  e :=  e(t), \quad  x :=  x(t), \quad  \dot{e} := \dot{e}(t) \quad \textrm{and} \quad   \dot{x} := \dot{x}(t). \]
Observe that 
\begin{align*}
  &   \frac{d}{dt} \,  \ln \frac{\left(  -\ppie{i-1}{e(t)}(x(t))\right)^{n-i}}{\left(   \ppie{i+1}{e(t)}(x(t)) \right)^{n-i+1}}  \\ 
& \quad  = (n-i) \,  \smfrac{d}{dt} \ln \left(  -\ppie{i-1}{e(t)}(x(t)) \right)  - (n-i+1) \,  \smfrac{d}{dt} \ln \ppie{i+1}{e(t)}(x(t))  \\
    & \quad =  \smfrac{n-i}{\ppie{i-1}{e}(x)} \left(   D\ppie{i-1}{e}(x)[ \dot{x} ] 
 + D_e (\ppie{i-1}{e}(x))[\dot{e}] \right)  \\ 
& \quad \qquad - \smfrac{n-i+1}{\ppie{i+1}{e}(x)}  \left(   D\ppie{i+1}{e}(x)[ \dot{x} ] + D_e (   \ppie{i+1}{e}(x) ) [\dot{e}] \right)    \; . 
\end{align*}
Using Lemma~\ref{t.h.k}  to substitute for $ \smfrac{n-i}{\ppie{i-1}{e}(x)} D\ppie{i-1}{e}(x) - \smfrac{n-i+1}{\ppie{i+1}{e}(x)}  D\ppie{i+1}{e}(x) $ gives 
\begin{align*}
&   \frac{d}{dt} \,  \ln \frac{\left(  -\ppie{i-1}{e(t)}(x(t))\right)^{n-i}}{\left(   \ppie{i+1}{e(t)}(x(t)) \right)^{n-i+1}} \\ 
&  \quad = \, - \, \frac{\ppie{i+1}{e}(x)}{i \ppie{i-1}{e}(x)} \left( D^2q_{e}(x)[ \dot{x}, \dot{x} ]  - \left(   \frac{c^* \dot{x} }{c^*(e-x )} \right)^2 \right) \\
& \quad \qquad + (n-i) \frac{D_e (\ppie{i-1}{e}(x))[ \dot{e}] }{\ppie{i-1}{e}(x)} - (n-i+1)  \frac{D_e (\ppie{i+1}{e}(x))[ \dot{e}] }{\ppie{i+1}{e}(x)} \quad \textrm{(using $ y^*(t) A \dot{x}(t) = 0$)} \\
& \quad \leq \frac{\ppie{i+1}{e}(x)}{i \ppie{i-1}{e}(x)} \left(   \frac{c^* \dot{x}  }{c^*(e-x)} \right)^2 + (n-i) \frac{D_e (\ppie{i-1}{e}(x))[ \dot{e}] }{\ppie{i-1}{e}(x)} - (n-i+1)  \frac{D_e (\ppie{i+1}{e}(x))[ \dot{e}] }{\ppie{i+1}{e}(x)} \; , 
\end{align*}
where the inequality is due to the combination of $ \ppie{i+1}{e}(x)/\ppie{i-1}{e}(x)  $ being negative (according to (\ref{e.c.d}) and (\ref{e.d.b})) and $ D^2q_e(x) $ being negative semidefinite (Theorem~\ref{t.g.b}). 

To complete the proof, simply substitute the following expressions which result from use of $ \dot{e} = x - e $, (\ref{e.c.f}) and (\ref{e.h.a}): 
\begin{align*} 
  D_e (\ppie{i-1}{e}(x))[ \dot{e}] & = (i-1) D\ppie{i-2}{e}(x) [ x -e ] \\
                          & = (i-1) \left(  (n-i+2) \ppie{i-2}{e}(x) - \ppie{i-1}{e}(x) \right)  \; , \\ & \\
 D_e (\ppie{i+1}{e}(x))[ \dot{e}]   & = (i+1) D\ppie{i}{e}(x)[ x - e ] \\
&  = (i+1) \left(   0 - \ppie{i+1}{e}(x) \right)   
\end{align*}
-- substitute, also, 
\begin{align*}
\frac{c^* \dot{x}  }{c^*(e-x)}  & = Dq_e(x)[ \dot{x}] \qquad \textrm{(first-order condition arising from Corollary \ref{t.g.d})} \\
& = \frac{D\ppie{i}{e}(x)[ \dot{x}] }{\ppie{i+1}{e}(x)}   \qquad \textrm{(using $ \ppie{i}{e}(x) = 0 $)} \\
& =  - \, \frac{ D_e( \ppie{i}{e}(x)) [\dot{e}] }{\ppie{i+1}{e}(x)}   \qquad \textrm{(because  $ \smfrac{d}{dt} \ppie{i}{e(t)}(x_{e(t)}) = 0 $)} \\
& = -i \, \frac{ D\ppie{i-1}{e}(x)[x-e] }{\ppie{i+1}{e}(x)} \\
& = -i \, \frac{ (n-i+1) \ppie{i-1}{e}(x) - 0 }{\ppie{i+1}{e}(x)} \; .  
\end{align*}
\hfill $ \Box $ 

\section{{\bf Proof of Part II of the Main Theorem}}  \label{s.i}

Recall the optimization problem dual to $ \hp $:
  \[  
 \left. \begin{array}{rl}
\sup_{y^*,s^*} & y^*b \\
\mathrm{s.t.} & y^* A + s^* = c^* \\
& s^* \in \Lambdap^* \end{array} \quad \right\} \, \hp^* \; ,  
\] where $ \Lambdap^* $ is the cone dual to $ \Lambdap $.  Recall that just as a matter of definitions, the optimal value of $ \hp^* $ satisfies $ \val^* \leq \val $ (``weak duality''), where $ \val $ is the optimal value of $ \hp $.

Recall that a pair $ (y^*,s^*) $ satisfying the constraints is said to be ``strictly'' feasible if $ s^* \in \int(\Lambdap^*) $ (interior).  

For $ e \in \swath(i) \setminus \core(i) $, define
\[  \sie{i}{e} :=  \frac{c^*(e-x) }{\ppie{i+1}{e}(x)} D \ppie{i}{e}(x) \quad \textrm{where } x = \xie{i}{e} \; . \]

We now restate and prove \hyperlink{targ_main_thm_part_two}{Part II} of the Main Theorem.
\begin{thm}  \label{t.i.a}
Assume $ 1 \leq i \leq n-2 $  and let $ \{ e(t): 0 \leq t < T \} $ be a maximal trajectory for the dynamics $ \dot{e}(t) = \xie{i}{e(t)} - e(t) $ starting at $ e(0)  \in \swath(i) \setminus \core(i) $.  Then $ y^*A + \sie{i}{e(t)} = c^* $ has a unique solution $ y^* = \yie{i}{e(t)} $, and  the pair $ (\yie{i}{e(t)},\sie{i}{e(t)}) $ is strictly feasible for $ \hp^* $.  Moreover, 
\[   \yie{i}{e(t)}b  =  c^* \xie{i}{e(t)} \xrightarrow[t \rightarrow T]{} \val \] 
(in fact, increases to $ \val $ strictly monotonically) and the path $ t \mapsto (\yie{i}{e(t)},\sie{i}{e(t)}) $ is bounded.
\end{thm}
\begin{proof}   Assume $ e \in \swath(i) \setminus \core(i) $.  By Corollary~\ref{t.g.d}, $ \xie{i}{e} $ is optimal for the convex optimization problem
\[  \begin{array}{rl}
 \min_x & - \ln c^*(e-x) - \frac{\ppie{i}{e}(x)}{\ppie{i+1}{e}(x)} \\
 \mathrm{s.t.} & Ax = b \; , 
\end{array} \]
and hence there exists $ y = \yie{i}{e} $ satisfying the (rearranged) first-order condition
\[  A^*y + \sie{i}{e} = c \; ,\]
where, letting $ q_e := \ppie{i}{e}/\ppie{i+1}{e} $ and $ x_e = \xie{i}{e} $, 
\[ \sie{i}{e} := \left( c^*(e - x_e) \right)  Dq_e( x_e) = \frac{c^*(e-x_e) }{\ppie{i+1}{e}(x_e)} D \ppie{i}{e}(x_e), \]
(using $ \ppie{i}{e}(x_e) = 0 $). 
 Thus, to show the pair $ (\yie{i}{e}, \sie{i}{e}) $ is strictly feasible for $ \hp^* $, it suffices to show $ \sie{i}{e} \in \int(\Lambdap^*) $, that is, assuming $ z \in \Lambdap $, $ z \neq 0 $, it suffices to show $ Dq_e(x_e)[z] > 0 $.

Corollary~\ref{t.d.b}  shows $ ( \partial \Lambdaiep{i}{e}) \setminus \Lambdap $ contains no line segments of positive length other than those lying in rays $ \{ sx: s > 0 \} $.  Thus, the line segment connecting $ x_e $ to $ z $ lies entirely within $ \Lambdaiepp{i}{e} $ with the exception of the point $ x_e $ and possibly the point $ z $.  Hence,  $ z(s) := (1-s)z + sx_e $ satisfies $ q_e(z(s)) > 0 $ for $ 0 < s < 1 $.  Fix $ s $ strictly between 0 and 1.

Concavity of $ q_e $ (according to Theorem~\ref{t.g.b}) and $ q_e(x_e) = 0 $ imply $ q_e(z(s))  \leq Dq_e(x_e)[z(s) - x_e] $, whereas  $ q_e(x_e) = 0 $ and homogeneity of $ q_e $ give  $ Dq_e(x_e)[x_e] = 0 $.  Thus,
\[  0 < q_e(z(s)) \leq  Dq_e(x_e)[z(s)-x_e] = (1-s) Dq_e(x_e)[z] \; , \]
completing the proof that $ (\yie{i}{e}, \sie{i}{e}) $ is strictly feasible for $ \hp^* $. Additionally observing
\begin{align*}
  \yie{i}{e} b & = \yie{i}{e} A x_e  \\
               & =  c^* x_e - \sie{i}{e} x_e \\
               & = c^* x_e -  \left(  c^*(e-x_e) \right)  Dq_e(x_e)[x_e] \\
               & = c^* x_e \; , 
\end{align*}
we thus see for a maximal trajectory $ \{ e(t): 0 \leq t < T \} $, each pair $ (\yie{i}{e(t)}, \sie{i}{e(t)}) $ is strictly feasible for $ \hp^* $, and 
\begin{align*}
     \yie{i}{e(t)} b & = c^* \xie{i}{e(t)} \\
                     & \qquad \rightarrow \val \quad \textrm{strictly monotonically}  \quad \textrm{(by Proposition~\ref{t.h.b})} \; .
\end{align*}

It only remains to show the path $ t \mapsto (\yie{i}{e(t)}, \sie{i}{e(t)}) $ is bounded, for which it suffices to show $ t \mapsto \sie{i}{e(t)} $ is bounded (as $ A $ is surjective, by assumption). In turn, because $ e(0) \in \Lambdapp $ and $ \sie{i}{e(t)} \in \Lambdap^* $, it suffices to show the value $ \sie{i}{e(t)} e(0) $  is bounded from above independent of $ t $.  However, 
\begin{align*}
    \sie{i}{e(t)} e(0) & = c^*e(0) - \yie{i}{e(t)} A e(0) \\
                   & = c^* e(0) - \yie{i}{e(t)} b \\
                   & \qquad \qquad \rightarrow \, c^* e(0) - \val \; , 
\end{align*}
concluding the proof.  \end{proof} 

\section{{\bf  Proof of Theorem 6}}  \label{s.j}

We have now finished our analysis of the trajectories $ t \mapsto e(t) $ (and paths $ t \mapsto \xie{i}{e(t)} $) arising from the differential equation $ \dot{e}(t) = \xie{i}{e(t)} - e(t) $ on the set $ \swath(i) \setminus \core(i) $.  However, in order to have a relatively complete picture of all of $ \swath(i) $, results regarding the structure of $ \core(i) $ are needed.  That is the purpose of this section.

Specifically, just as the dynamics $ \dot{e}(t) = \xie{i}{e(t)} - e(t) $ on $ \swath(i) \setminus \core(i) $ leads to optimality, it would be nice as a matter of formalism to know that if $ e \in \core(i) $ and $ x \in \opt $, then moving $ e $ towards $ x $ results in a point $ e + t(x-e) $ ($ 0 < t < 1 $) also in $ \core(i) $, thus retaining optimality.   

\begin{thm} \label{t.j.a}
Assume $ 1 \leq i \leq n-2 $.  If $ e \in \core(i) $ and $ x \in \opt $ then
\[    \{ e + t(x-e) \in \Lambdapp : t \in \mathbb{R} \} \subseteq \core(i) \; . \]
\end{thm} 
\noindent 
Extending the theorem to include $ i = 0 $ is trivial, because $ \core(0) $ is precisely the interior of $ \feas $.  It also is easy to extend the theorem to $ i = n-1 $, simply because $ \core(n-1) = \emptyset $.\footnote{To see $ \core(n-1) = \emptyset $, recall that  $  (\partial \Lambdaiep{n-1}{e}) \cap \Lambdap $ is precisely the lineality space of $ \Lambdap $ (by Theorem~\ref{t.d.a}(C)).  Thus, since $ \Lambdap $ is regular (by assumption) and since the origin is infeasible    (because $ b \neq 0 $), the boundary of $ \feas $  cannot intersect the boundary of $ \feasie{n-1}{e} $; in particular, it cannot happen that $ \opt $ intersects $ \optie{n-1}{e} $.}  We state the theorem only for $ 1 \leq i \leq n-2 $ to avoid having to be mindful of the special cases $ i = 0, n-1 $ during the course of the proof.

By bootstrapping the theorem, the following elementary corollary provides some additional insight into the structure of $ \core(i) $.  The corollary is a restatement of Theorem \hyperlink{targ_thm_six}{6}.
\begin{cor} \label{t.j.b}
 If $ e \in \core(i) $ and $ {\mathcal A} $ is the smallest affine space containing both $ e $ and $ \opt $, then 
\[  {\mathcal A} \cap \Lambdapp  \subseteq \core(i) \; . \] 
\end{cor}
\begin{proof} 
If $ \opt $ contains only a single point, then the corollary is nothing more than a restatement of the theorem.  Thus, assume $ \opt $ contains more than a single point, and fix $ \tilde{x}     \in \relint(\opt) $.  Then
\[ 
{\mathcal A} = \{ \tilde{x} + \alpha \, ( e - \tilde{x}) + \beta \, (x - \tilde{x}): \alpha \in \mathbb{R}, \, \beta \geq 0 \textrm{ and } x \in  \opt \} \; . 
\]

Assume $ y \in {\mathcal A} \cap \Lambdapp $, and fix $ \alpha \in \mathbb{R} $, $ \beta \geq 0 $ and $ x \in \opt $ satisfying
\[  y  = \tilde{x} + \alpha \, ( e - \tilde{x}) + \beta \, (x - \tilde{x}) \; . \]
To complete the proof, we show $ y \in \core(i) $.

Since $ y $ and $ e $  lie in the interior of $ \feas $, we have $ c^* y, c^* e > \val $ ($ = c^* \tilde{x}, c^* x $), from which follows $ \alpha > 0 $.  Thus, the point 
\[  z :=  \smfrac{\alpha }{\alpha + \beta } e + \smfrac{\beta }{\alpha + \beta } \, x \]
is well-defined, is a convex combination of $ e $ and $ x $, and lies in $ \Lambdapp $ (because $ e \in \Lambdapp $, $ x \in \Lambdap $ and $ \alpha > 0 $).  As $ e \in \core(i) $ and $ x \in \opt $, Theorem~\ref{t.j.a}  implies $ z \in \core(i) $.  

However, $ y = z + t(\tilde{x} -z) $ for $ t = \alpha + \beta - 1 $.  Hence, applying Theorem~\ref{t.j.a}  again, now with $ z $ (resp., $ \tilde{x} $) in place of $ e $ (resp., $ x $), we find $ y \in \core(i) $.  \end{proof}    

Before turning to the proof of Theorem~\ref{t.j.a}, we recall the conjecture made in \S\ref{s.b}.
\begin{conj}
$ \core(i) $ is convex. 
\end{conj}
\noindent It is highly unclear whether resolving the conjecture positively could have any relevance to algorithms.  Nonetheless, a positive resolution would be interesting from a structural perspective.  

Now we begin the process of proving Theorem~\ref{t.j.a}.  For motivation, we first consider a reasonably-general case for which the proof is rather straightforward, and afterwards move to the proof for the truly general case.  That the proof is straightforward in one setting but not in full generality is reminiscent of the proof of Part I of the Main Theorem, where the case of ``$ T = \infty $'' was vastly easier than the case ``$ T < \infty $'' (i.e., \S\ref{s.h.a}  was a breeze compared to \S\ref{s.h.b}).

The straightforward case relies on the following lemma.
\begin{lemma}  \label{t.j.c}
If $ 0 \leq i \leq n-2 $ and $ x \in \partial \Lambdaiep{i}{e} $, then
\[ D\ppie{i}{e}(x) = 0 \quad \Leftrightarrow \quad x \in \partial \Lambdaiep{i+1}{e}  \; . \]
\end{lemma}
\begin{proof} This is Lemma 7 in \cite{renegar}. \end{proof} 

Here is the straightforward case:
\begin{thm}  \label{t.j.d}
Assume $ 1 \leq i \leq n-2 $ and $ e \in \core(i) $.  If $ x \in \opt \cap \Lambdaiepp{i+1}{e} $ then
\[    \{ e + t(x-e) \in \Lambdapp : t \in \mathbb{R} \} \subseteq \core(i) \; . \]
\end{thm}
\begin{proof} Recall Corollary~\ref{t.g.d}, which shows that for every $ \bar{e} \in \Lambdapp $ and $ 0 \leq i \leq n-2 $, the points lying in $ \optie{i}{ \bar{e}} \setminus \partial \Lambdaiep{i+1}{\bar{e}} $ are precisely the optimal solutions to the linearly constrained, convex optimization problem
\[   \begin{array}{rl}
    \min_{x \in  \Lambdaiepp{i+1}{\bar{e}}}  & - \ln c^*( \bar{e} - x) - \frac{\ppie{i}{\bar{e} }(x)}{\ppie{i+1}{\bar{e}}(x)} \\
    \textrm{s.t.} & Ax = b \; , \end{array} \]
that is, for every $ \bar{e} \in \Lambdapp $, letting $ \qie{i}{ \bar{e}} := \ppie{i}{ \bar{e}}/\ppie{i+1}{ \bar{e}} $,   
\begin{equation}  \label{e.j.a}
\left. \begin{array}{c}
  x \in \optie{i}{\bar{e} } \cap \Lambdaiepp{i+1}{\bar{e}}  \\ \Leftrightarrow \\ 
  \smfrac{1}{c^*( \bar{e}-x) } \, c^* - D \qie{i}{ \bar{e}}(x) = y^* A \quad \textrm{for some $ y^* $} \\
 Ax = b \\
  x \in \Lambdaiepp{i+1}{ \bar{e}} \; .       
\end{array} \quad \right\} \end{equation}

Now assume $ x \in \opt \cap \Lambdaiepp{i+1}{e} $, where $ e \in \core(i) $ and $ 0 \leq i \leq n-2 $.  Then, since $ \Lambdap \cap \partial \Lambdaiep{j}{ \bar{e}} $ ($ 0 \leq j \leq n-1 $) is independent of $ \bar{e}  \in \Lambdapp $ (by Theorem~\ref{t.d.a}(B)), we have for all $ \bar{e}   \in \Lambdapp $ that
\begin{gather} 
  x \in \Lambdaiepp{i+1}{ \bar{e}} \; ,   \label{e.j.b} \\
   \ppie{j}{ \bar{e}}(x) = 0 \quad \textrm{for all $ 0 \leq j \leq i $} \; , \label{e.j.c}
\end{gather}
and, using Lemma~\ref{t.j.c},
\begin{equation}  \label{e.j.d}
     D \ppie{j}{ \bar{e}}(x) = 0 \quad \textrm{for all $ 0 \leq j \leq i-1 $}   \; . 
\end{equation} 
Moreover, by (\ref{e.j.a}), there exists $ y^* $ satisfying
\[ 
\smfrac{1}{c^*( e-x) } \, c^* - D \qie{i}{ e}(x) = y^* A \; , 
\]
that is, as $ \ppie{i}{e}(x) = 0 $, there exists $ y^* $ satisfying
\begin{equation}  \label{e.j.e}
\smfrac{1}{c^*( e-x) } \, c^* - \smfrac{1}{\ppie{i+1}{e}(x)} D \ppie{i}{ e}(x) = y^* A \; .
\end{equation}

Fix $ t $ for which $  e + t(x-e)  \in \Lambdapp $ (thus, $ - \infty < t < 1 $).  To prove the theorem, it suffices to show $ x \in \optie{i}{e + t(x-e)} $, because then, $ \opt \cap \optie{i}{e+t(x-e)} \neq \emptyset $, and hence, $ \opt = \optie{i}{e+t(x-e)} $ (by Corollary~\ref{t.d.b}(B)).  Thus, in light of (\ref{e.j.b}), we see from (\ref{e.j.a})  that to prove the theorem, it suffices to show
\[ 
   \smfrac{1}{(1-t) \, c^*(e -x) } \, c^* - D \qie{i}{ e + t(x-e) }(x) = w^* A \quad \textrm{for some $ w^* $} \; , 
\] 
that is, as $ \ppie{i}{ e + t(x-e)}(x) = 0 $ (by (\ref{e.j.c})), it suffices to show
\begin{equation}  \label{e.j.f}
   \smfrac{1}{(1-t) \, c^*(e -x) } \, c^* - \smfrac{1}{\ppie{i+1}{ e + t(x-e) }(x)} D \ppie{i}{ e + t(x-e) }(x) = w^* A \quad \textrm{for some $ w^* $} \; .  
\end{equation}
However,    
  \begin{align*}
     D \ppie{i}{e+t(x-e)}(x) & = D^{i+1}p(x)[\underbrace{(1-t)e + tx, \ldots, (1-t)e + tx}_{\textrm{$ i $ times}}] \\ & = \sum_{j = 0}^i \binom{i}{ j }(1-t)^{ j} t^{i - j} D^{i+1}p(x)[\underbrace{e, \ldots, e}_{\textrm{$ j $ times}}, \underbrace{x, \ldots, x}_{\textrm{$ i - j $ times}}] \\
        & = \sum_{j = 0}^i \binom{i}{ j }(1-t)^{ j} t^{i - j} D^{i+1-j }\ppie{j}{e}(x)[\underbrace{x, \ldots, x}_{\textrm{$ i - j $ times}}] \\
         & = \sum_{j = 0}^i \binom{i}{ j }(1-t)^{ j} t^{i-j} \, \frac{(n-j-1)!}{(n-i-1)!} D \ppie{j}{e}(x) \quad \textrm{(by (\ref{e.c.f}))} \; \\
    & =  (1-t)^i D \ppie{i}{e}(x) \quad \textrm{(by (\ref{e.j.d}))} \; ,  
\end{align*}
and, similarly,
\begin{align*}
  \ppie{i+1}{e+t(x-e)}(x) & =  \sum_{j = 0}^{i+1} \binom{i+1}{ j }(1-t)^{ j} t^{i+1-j} \, \frac{(n-j)!}{(n-i-1)!} \ppie{j}{e}(x) \\
  & =  (1-t)^{i+1} \ppie{i+1}{e}(x) \quad \textrm{(by (\ref{e.j.c}))} \; . 
\end{align*}
Hence, $ w^* = \smfrac{1}{1-t} \, y^* $ satisfies (\ref{e.j.f})  (where $ y^* $ is as in (\ref{e.j.e})), completing the proof. \end{proof} 

In the course of the proof for the ``truly general'' case, there is a crucial outside result to which we refer, and thus before formally beginning the proof, we explain the result.
 
For $ e \in \Lambdapp $ and an arbitrary point $ x $, let $ M_e(x) $ denote the number of non-negative roots of $ t \mapsto p(x+te) $.  The result is that $ M_e(x) $ is independent of $ e \in \Lambdapp $, as was stated in \cite{hl} (as Theorem 2.12). (In \cite{renegar}, the fact was (essentially) established {\em only}  for $ x \in \Lambdap $.) 

Perhaps worth recording is that the independence easily follows from a most useful tool in the hyperbolic polynomial literature: 
\hypertarget{hv_thm}{}
\begin{rem}  Assume $ e \in \Lambdapp $.  For any points $ x $ and $ z $, there exist $ n \times n $ symmetric matrices $ X $ and $ Z $ such that
\[  (r,s,t) \mapsto p(rx + sz + te) = p(e) \, \det( rX + sZ + tI) \; . \]
\end{rem}
\noindent This formulation of the theorem comes from \cite{lpr}, where it was obtained by straightforward homogenization of the original result in \cite{hv}.  (The initial importance of the homogeneous version was that it affirmatively settled the ``Lax Conjecture'' -- see \cite{lpr} for discussion.)  (See \cite{branden} for negative results on possibilities of extensions to more than three variables $ r $, $ s $, $ t $.) 

\addtocounter{thm}{1}

\begin{cor}  \label{t.j.f}
For $ e \in \Lambdapp $ and arbitrary $ x $, let $ M_e(x) $ denote the number of non-negative roots for the polynomial $ t \mapsto p(x+te) $. The value $ M_e(x) $ is the same for all $ e \in \Lambdapp $.
\end{cor}
\begin{proof}  
First observe that because all roots of univariate polynomials vary continuously in the coefficients so long as the leading coefficient does not vanish, it suffices to show $ \mult_e(x) $ is independent of $ e \in \Lambdapp $, where $ \mult_e(x) $ is the multiplicity of 0 as a root of $ t \mapsto p(x+te) $.  However, for $ e, z  \in \Lambdapp $, using the Helton-Vinnikov Theorem,
\[  \mult_e(x) = n - \mathrm{rank}(X) = n - \mathrm{rank}(Z^{-1/2}XZ^{-1/2}) = \mult_{z}(x) \; , \]
and thus the corollary is established. \end{proof}

\noindent 
{\bf  {\em Proof of Theorem~\ref{t.j.a}.}}  Fix $ e \in \core(i) $, where $ 1 \leq i \leq n-2 $.  
 Let $ x $ denote a point in $  \opt $ ($ = \optie{i}{e} $).

Let $ e(t) :=  e + t(x-e) $, and let $ I  $ denote the open interval consisting of all $ t \in \mathbb{R} $ for which $ e(t) \in \Lambdapp $ -- thus, $ I \subseteq (- \infty, 1) $.  

Since $ \Lambdap \cap \partial \Lambdaiep{i}{\bar{e}} $ is independent of $ \bar{e} \in \Lambdapp $ (by Corollary~\ref{t.d.b}(A)), and since    $ \optie{i}{e} = \opt $, we have
\begin{equation}  \label{e.j.g}
\opt \subset \partial \Lambdaiep{i}{e(t)} \textrm{ for all $ t \in I $} \; . 
\end{equation}

Let $ {\mathcal D}  := \{ d: Ad = 0 \textrm{ and } c^*d = 0 \} $, and let $ x + {\mathcal D}  $ be the set consisting of all points $ x + d $ where $ d \in {\mathcal D}  $. 

Since $ \relint(\feasie{i}{e(t)}) = (x + \{d: Ad = 0 \}) \cap \Lambdaiepp{i}{e(t)} $ and $ x \in \opt \subset \feasie{i}{e(t)} $, we have 
\[  x \in \optie{i}{e(t)} \quad \Leftrightarrow \quad  (x + {\mathcal D} ) \cap \Lambdaiepp{i}{e(t)} = \emptyset \; . \]
On the other hand, $ x $ is in $ \optie{i}{e(t)} $, as well as in $ \opt $, if and only if $ e(t) \in \core(i) $ (by Corollary~\ref{t.d.b}(B)).  Consequently,
\[ 
e(t) \in \core(i) \quad \Leftrightarrow \quad  (x + {\mathcal D} ) \cap \Lambdaiepp{i}{e(t)} = \emptyset \; . \]
Thus, our goal -- proving $ e(t) \in \core(i) $ for all $ t \in I $ -- will be accomplished if we prove $ (x + {\mathcal D} ) \cap \Lambdaiepp{i}{e(t)} = \emptyset $ for all $ t \in I $. Hence, fixing arbitrary $ d \in {\mathcal D} $, letting $ \ray(d) := \{ sd: s \geq 0 \} $ and
\[  I(d) := \{ t \in I:  (x + \ray(d)) \cap \Lambdaiepp{i}{e(t)} =  \emptyset \} \; , \]
our goal is to show $ I(d) = I $.

Recall, however, the characterization (\ref{e.c.b}) now applied to $ \Lambdaiepp{i}{e(t)} $:
\begin{equation}  \label{e.j.h}
\Lambdaiepp{i}{e(t)} := \{ x: \ppie{j}{e(t)}(x) > 0 \textrm{ for all $ j = i,  \ldots, n-1 $} \} \; . 
\end{equation}  
Clearly, then, $ I \setminus I(d) $ is an open subset of the open interval $ I $.  Hence, since $ I(d) \neq \emptyset $ (indeed, $ 0 \in I(d) $), to complete the proof of the theorem, it suffices to show that $   I(d) $ is open.
As $ e $ is only assumed to be an  element of $ \core(i) $ -- in particular, $ e $ could be replaced with any $ e(t) $ that happens to be in $  \core(i) $ -- our goal has become: 
\begin{equation}  \label{e.j.i}
\textrm{Show $ I(d) $ contains an open interval including $ 0 $.} 
 \end{equation}
  
For non-negative integers $ j $ and $ k $, define
\[           \alpha_{jk}(t) := D^k \ppie{j}{e(t)}(x)[\underbrace{d, \ldots, d}_{\textrm{$ k $ times}}]   \]
and
\[  \beta_j(t) := \begin{cases} 0 & \textrm{if $ \alpha_{jk}(t) = 0 $ for all $ k $ ;} \\
\alpha_{j,k(j,t)} & \textrm{otherwise, where $   k(j,t) := \min \{ k: \alpha_{jk}(t) \neq 0 \} $  .} \end{cases} \]

Note that if $ \beta_j(t) = 0 $, then the polynomial $ s \mapsto \ppie{j}{e(t)}(x+sd) $ is identically zero, whereas if $  \beta_j(t) \neq 0 $, the polynomial evaluated at small, positive $ s $ has the same sign as $ \beta_j(t) $.  Thus, since $ x $ lies in the convex set $ \Lambdaiep{i}{e(t)} $, we have by the characterization (\ref{e.j.h})  that
\[  (x + \ray(d)) \cap \Lambdaiepp{i}{e(t)} \neq \emptyset \quad \Leftrightarrow \quad \beta_j(t) > 0 \textrm{ for all $ j = i, \ldots, n-1 $} \; . \]
Our goal  (\ref{e.j.i})  has now been reduced to:
\begin{equation}  \label{e.j.j}
 \textrm{Show }    \exists \epsilon > 0 \textrm{ such that } |t| < \epsilon \, \Rightarrow \, \beta_j(t) \leq 0 \textrm{ for some $ j \in \{ i, \ldots, n-1 \}  $} \; . 
\end{equation}
For this we consider two cases.

First, assume $ \beta_j(0) \geq  0 $ for all $ j = i, \ldots, n-1 $.  Then $ x + sd \in \feasie{i}{e} $ for all sufficiently small, positive $ s $.  Thus, since $ x \in \optie{i}{e} = \opt $ and $ c^*d = 0 $, we have $ x + sd \in \opt  $ for all sufficiently small, positive $ s $.  Hence, by (\ref{e.j.g}) , for each $ t \in  I $, the univariate polynomial $ s \mapsto \ppie{i}{e(t)}(x+se) $ has value zero on an open interval, and thus is identically zero, so $  \beta_i(t) =  0 $ for all $ t \in I $, (more than) accomplishing (\ref{e.j.j}) for the first case.

Now consider the remaining case, that is, assume $ \beta_j(0) < 0 $ for some $ j   \in \{ i, \ldots, n-1 \} $; fix such a $ j $.  To accomplish (\ref{e.j.j}) , of course it suffices to show for this fixed value of $ j $ that $ \beta_{j}(t) < 0 $ for all $ t $ in an open interval containing $ 0 $. In turn, by definition of $ \beta_j(t) $, it suffices to: 
\begin{equation}  \label{e.j.k}
  \textrm{Show }  \exists \epsilon>0 \textrm{ such that } |t| < \epsilon \quad  \Rightarrow  \quad 
D^{k(j,0)} \ppie{j}{e(t)}(x)[\underbrace{d, \ldots, d}_{\textrm{$ k(j,0) $ times}} ] < 0
\end{equation}
and
\begin{equation}  \label{e.j.l}
 \textrm{show }  (t \in I) \wedge (k < k(j,0)) \quad \Rightarrow \quad D^{k} \ppie{j}{e(t)}(x)[\underbrace{d, \ldots, d}_{\textrm{$ k $ times}} ] = 0 \; . 
\end{equation}
 
The existence of $ \epsilon $ as in (\ref{e.j.k})  is simply a matter of continuity and $ \beta_j(0) < 0 $.
For accomplishing (\ref{e.j.l}) , observe 
\begin{align*}
     D^k \ppie{j}{e(t)}(x) & = D^{k+j}p(x)[\underbrace{(1-t)e + tx, \ldots, (1-t)e + tx}_{\textrm{$ j $ times}}] \\
        & = \sum_{\ell = 0}^j \binom{j}{ \ell }(1-t)^{ \ell} t^{j - \ell} D^{k+j- \ell }\ppie{\ell}{e}(x)[\underbrace{x, \ldots, x}_{\textrm{$ \ell - j $ times}}] \\
         & = \sum_{\ell = 0}^j \binom{j}{ \ell }(1-t)^{ \ell} t^{j - \ell} \, \frac{(n-k-\ell)!}{(n-k-j)!} D^k \ppie{\ell}{e}(x) \quad \textrm{(using (\ref{e.c.f}))} \; . 
\end{align*}
Consequently, (\ref{e.j.l})  is immediately accomplished by the following proposition, thus concluding the proof of the theorem (except for proving the proposition).  \hfill $ \Box $ 

\begin{prop}  \label{t.j.g}
Assume $ e \in \Lambdapp $, $ x \in \Lambdap $ and let $ d $ be a vector.  Assume non-negative integers $ j $ and $ k $ satisfy $ D^k \ppie{j}{e}(x) [\underbrace{d, \ldots, d}_{k \textrm{ times}}] \neq 0 $;  let $ k(j) $ be the smallest such $ k $ for $ j $ and assume $ k(j) > 0 $.  

Then
\[  D^{k } \ppie{  \ell }{e}(x)[\underbrace{d, \ldots, d}_{ k  \textrm{ times}}] = 0 \quad \textrm{for all $ \ell  = 0, \ldots, j $ and $ k = 0,\ldots, k(j)-1 $ } \; . \]
\end{prop}

The proof of the proposition makes use of the following lemma.
  
For $ e \in \Lambdapp $ and arbitrary $ z $, recall that $ M_e(z) $ denotes the number of non-negative roots of $ t \mapsto p(z+te) $.

\begin{lemma}  \label{t.j.h}
Assume $ e, x \in \Lambdapp $, let $ d \neq 0 $, and  let $ s_1 \leq s_2 \leq \cdots \leq s_k $ denote the positive roots (including multiplicities) of $ s \mapsto p(x+sd) $.  Then, for every $ \bar{s}   \geq 0 $, 
\[   M_e(x+\bar{s}  d) = \# \{ i: s_i \leq \bar{s}   \} \; . \]
\end{lemma}
\begin{proof}  Because $ M_e(x+\bar{s} d) $ is independent of $ e \in \Lambdapp $ (by Corollary~\ref{t.j.f}), in proving the lemma we may assume $ e = x $, in which case we need only consider the hyperbolic polynomial obtained by restricting $ p $ to the subspace spanned by $ x $ and $ d $.  However, every hyperbolic polynomial in two variables of degree $ n $  is of the form $ (y_1,y_2) \mapsto \prod_{j=1}^n a_j^T y $ for some vectors $ a_j \in \reals^{2} $ (this easily follows from two facts: (i) every complex homogeneous polynomial in two variables of degree $ n $ is of the form $ (y_1,y_2) \mapsto \prod_{j=1}^n (a_{j,1}y_1 + a_{j,1}y_2) $ for some $ a_j \in \mathbb{C}^2 $; (ii) if $ a \in \mathbb{C}^2  $ is not a (complex) multiple of a real vector, then $ \{ (y_1,y_2) \in \mathbb{R}^2 : a_1y_1 + a_2y_2 = 0 \} $ contains only the origin).   

Thus, we need only consider hyperbolic polynomials $ p(y_1, y_2) = \prod_{j=1}^n a_j^T y $ -- where $ a_j \in \mathbb{R}^2 $ -- and $ x = (x_1, x_2) $ which for all $ j $ satisfies $ a_j^T x \neq  0 $. For $ (d_1, d_2) \neq (0,0) $, our goal is to show for $ \bar{s}   \geq 0 $ that the number of roots $ 0 < \hat{s}  \leq \bar{s}   $ (counting multiplicities) for the univariate polynomial 
\begin{equation}  \label{e.j.m}
  s \mapsto \prod_{j=1}^n a_j^T (x + sd) 
\end{equation}
  is the same as the number of roots $ \hat{t} \geq 0 $  (counting multiplicities) for 
\begin{equation}  \label{e.j.n}
  t \mapsto \prod_{j=1}^n a_j^T (x + \bar{s}d + t x ) \; . 
\end{equation}
   Of course, however, $ \hat{s} > 0  $ is a root of (\ref{e.j.m})  if and only if $ \hat{s}  = \bar{s}/(1+\hat{t} ) $  for some root $ \hat{t} \geq 0  $ of (\ref{e.j.n}).  The lemma follows.
\end{proof}

\noindent {\bf  {\em Proof of Proposition~\ref{t.j.g}}} Assume $ e \in \Lambdapp $.  For arbitrary $ z $ and $ 0 \leq \ell \leq n - 1 $, let  $ \mie{\ell}{e}(z) $ denote the number of non-negative roots (counting multiplicities) for $ t \mapsto \ppie{\ell}{e}(x+te) $.  Additionally, for a vector $ d $ and value $ s \geq  0 $, let $ \nie{\ell}{e}(x,d,s) $ denote the number of roots (counting multiplicities) in the closed interval $ [-s, s] $ for the univariate polynomial $ \bar{s}  \mapsto \ppie{\ell}{e}(x+ \bar{s} d) $; if the univariate polynomial is identically zero, let $ \nie{\ell}{e}(x,d,s) =  \infty $.

If $ x \in \Lambdapp $ (hence $ x \in \Lambdaiepp{\ell}{e} $ for all $ 0 \leq \ell \leq n-1 $), $ d \neq 0 $ and $ s \geq 0 $, Lemma~\ref{t.j.h}  can be applied with $ \ppie{\ell}{e} $ in place of $ p $, and applied with $ -d $ as well as with $ d $, yielding
\begin{equation}  \label{e.j.o}
      \nie{\ell}{e}(x,d,s) =  \mie{\ell}{e}(x+sd) + \mie{\ell}{e}(x-sd) \; . 
\end{equation} 
On the other hand, for any $ z $, the interlacing of the roots of $ t \mapsto \ppie{\ell}{e}(z+te) $ and its derivative $ t \mapsto \ppie{\ell +1}{e}(z+te) $ gives $ \mie{\ell}{e}(z) \geq \mie{\ell +1}{e}(z) $, and thus,  
\begin{equation}  \label{e.j.p}
   \ell \leq j \quad \Rightarrow \quad  \mie{\ell}{e}(z) \geq \mie{j}{e}(z) \; . 
\end{equation}
From (\ref{e.j.o})  and (\ref{e.j.p})  follows
\begin{equation}  \label{e.j.q}
   \left( x \in \Lambdapp \right) \wedge \left( \ell \leq  j \right)  \quad \Rightarrow \quad  \nie{\ell}{e}(x,d,s) \geq \nie{j}{e}(x,d,s) \; .  
\end{equation}
 
Assume, now, that $ x $, $ d $, $ j $ and $ k(j) $ satisfy the hypothesis of the proposition.  Observe that definitions readily give $ k(j) = \nie{j}{e}(x,d,0) $.  Moreover, proving the proposition amounts precisely to showing 
\begin{equation}  \label{e.j.r}
      \nie{\ell}{e}(x,d,0) \geq \nie{j}{e}(x,d,0) \quad \textrm{for $ \ell = 0, \ldots, j $} \; . 
\end{equation}
  
For $ \epsilon > 0 $, let $ x( \epsilon ) := x + \epsilon e $, a point in $   \Lambdapp $.   Each polynomial $ \bar{s}  \mapsto \ppie{\ell}{e}( x( \epsilon)+ \bar{s}d)  $ ($ 0 \leq \ell \leq n-1 $) is not identically zero and has only real roots (indeed, by homogeneity, the roots are the reciprocals of the non-zero roots for $ t \mapsto \ppie{\ell}{e}(d + tx( \epsilon )) $).  Thus, since
\begin{equation}  \label{e.j.s}
\left. \begin{array}{c}
  \textrm{bounded roots of a univariate polynomial} \\ \textrm{vary continuously in the coefficients}\\
  \textrm{(so long as the polynomial does not become identically zero)}  \end{array} \quad \right\}
\end{equation}
we have for $ 0 \leq \ell \leq j $ that either $ \nie{\ell}{e}(x,d,0) = \infty $ (i.e., $ \bar{s}  \mapsto \ppie{\ell}{e}(x + \bar{s} d) \equiv 0 $), or
\begin{align*}
     \nie{\ell}{e}(x,d,0) & =  \lim_{ s \downarrow 0} \left(  \lim_{ \epsilon \downarrow 0} \left(  \nie{\ell}{e}(x( \epsilon),d,s) \right)  \right)  \\
       & \geq  \lim_{ s \downarrow 0} \left(  \lim_{ \epsilon \downarrow 0} \left(  \nie{j}{e}(x( \epsilon),d,s) \right) \right) \quad \textrm{(by (\ref{e.j.q}))} \\
     &  \geq  \lim_{ s \downarrow 0} \left(    \nie{j}{e}(x,d,s/2) \right)  \quad \textrm{(again using (\ref{e.j.s}))} \\
     & =  \nie{j}{e}(x,d,0) \; , 
\end{align*}
thereby establishing (\ref{e.j.r})  and hence completing the proof.  \hfill $ \Box $


\begin{thebibliography}{99}


\bibitem{bgls}
H. H. Bauschke, O. G\"{u}ler, A. S.  Lewis and H. S. Sendov, Hyperbolic polynomials and convex analysis, {\it Canadian Journal of Mathematics}  {\bf 53} (2001) no.~3, 470--488.  

\bibitem{bb}
J. Borcea and P. Br\"{a}nd\'{e}n, Multivariate P\'{o}lya-Schur classification problems in the Weyl algebra, {\em Proceedings of the London Mathematical Society} {\bf 101} (2010) no.~1, 73-104.

\bibitem{branden}
P. Br\"{a}nd\'{e}n, Obstructions to determinantal representability, {\em Advances in Mathematics} {\bf 226} (2011) no.~2, 1202-1212.  

\bibitem{chua}
C.B. Chua, The primal-dual second-order cone approximations algorithm for symmetric cone programming, {\it Foundations of Computational Mathematics} {\bf 7} (2007), no.~3, 271-302. 

\bibitem{fm}
A.V. Fiacco and G.P. McCormick, {\it Nonlinear Programming: Sequential Unconstrained Minimization Techniques}, SIAM, Philadelphia, PA, 1990.


\bibitem{garding}
 L. G\aa rding, 
 An inequality for hyperbolic polynomials, 
 {\it Journal of Mathematics and Mechanics}   {\bf 8} (1959) no.~6, 957--965.

\bibitem{guler}
O. G\"uler, Hyperbolic polynomials and interior point methods for convex programming, {\it Mathematics of Operations Research}  {\bf 22} (1997), no.~2, 350--377. 

\bibitem{gurvits} L. Gurvits, Van der Waerden/Schrijver-Valiant like conjectures and stable (aka hyperbolic)
homogeneous polynomials: one theorem for all, {\it Electronic Journal of Combinatorics} {\bf 15}  (2008)
\#R66.

\bibitem{hl}
F.R. Harvey and H.B. Lawson Jr., Hyperbolic polynomials and the Dirichlet problem, preprint available at arXiv.org.

\bibitem{hv}
J. Helton and V. Vinnikov, Linear matrix inequality representations of sets, {\em Communications of Pure and Applied Mathematics} 60 (2007) 654-674.

\bibitem{lpr}
A.S. Lewis,
P.A. Parrilo,
M.V. Ramana, The Lax conjecture is true, {\it Proceedings of the American Mathematical Society}  {\bf 133} (2005) no.~9, 2495--2499.

\bibitem{nn}  Y. Nesterov\ and\ A. Nemirovski, {\it Interior-point polynomial algorithms in convex programming}, SIAM, Philadelphia, PA, 1994.


\bibitem{renegar}
J. Renegar, Hyperbolic programs, and their derivative relaxations, {\em Foundations of Computational Mathematics}  {\bf 6}  (2006) no.~1, 59-79.

\bibitem{zinchenko1}
Y. Zinchenko, {\it The local behavior of the Shrink-Wrapping algorithm for linear programming},
Ph.D. thesis, Cornell University, 2005.

\bibitem{zinchenko2}
Y. Zinchenko, On hyperbolicity cones associated with elementary symmetric polynomials,
{\it Optimization Letters} {\bf 2} (2008) no.~3, 389-402.

\bibitem{zinchenko3}
Y. Zinchenko, Shrink-wrapping trajectories for linear programming, preprint available at optimization-online.org.  
 \end{thebibliography}
\end{document}